\documentclass{amsart}

\usepackage{graphicx}
\usepackage{amssymb}
\usepackage{float}
\usepackage{subcaption}
\usepackage{tikz}
\usetikzlibrary{calc}
\usepackage{hyperref}

\newtheorem{theorem}{Theorem}[section]
\newtheorem{lemma}[theorem]{Lemma}
\newtheorem{prop}{Proposition}[section]
\newtheorem*{qcrit}{Q-Criterion}
\newtheorem{con}{Condition}
\newtheorem*{local}{Local Ergodic Theorem}
\newtheorem*{main}{Main Theorem}
\newtheorem*{abundance}{Abundance of least expanding points}
\newtheorem*{cs}{Chernov-Sinai ansatz}

\theoremstyle{definition}
\newtheorem{definition}[theorem]{Definition}

\theoremstyle{remark}

\numberwithin{equation}{section}

\begin{document}

\title[]{Certain systems of three falling balls satisfy the Chernov-Sinai Ansatz}

\author{Michael Tsiflakos}
\address{University of Vienna, Faculty of Mathematics.}
\curraddr{Oskar Morgensternplatz 1, 1090 Vienna.}
\email{michael.tsiflakos@univie.ac.at}
\thanks{\texttt{The author was supported by a Marietta-Blau and Marshall Plan Scholarship.}}

%Information for second author
%\author{Author Two}
%\address{Mathematical Research Section, School of Mathematical Sciences,
%Australian National University, Canberra ACT 2601, Australia}
\thanks{\textsf{Acknowledgements} I wish to cordially thank Maciej P. Wojtkowski for his help in 
outlining 
the difficulties of the problem and his hospitality during my visit to Opole in October 2017. 
Further, my gratitude is expressed to my advisor Henk Bruin and P\'eter B\'alint, Nandor Sim\'anyi, 
Domokos Sz\'asz and Imre P\'eter T\'oth for many helpful discussions. }

%    General info

\date{\today}


%\keywords{Differential geometry, algebraic geometry}

\begin{abstract}
The system of falling balls is an autonomous Hamiltonian system with a smooth 
invariant measure and non-zero Lyapunov exponents almost everywhere. Since almost three decades, 
the question of ergodicity is still open. The subject of this 
work is to contribute to the solution of the ergodicity conjecture for three 
falling balls with a specific mass ratio. The latter is executed in the following 
three points: First, we prove the Chernov-Sinai ansatz. Second, we prove that there 
is an abundance of least expanding points and, third, we explain that the proper alignment 
condition can still be verified and is actually pointwise equivalent to 
Chernov's transversality condition. It is of special interest, that for the aforementioned 
specific mass ratio, the configuration space can be unfolded to a billiard 
table, where the proper alignment condition holds.
\end{abstract}

\maketitle

\section{Introduction}
The system of falling balls was introduced by Wojtkowski \cite{W90a,W90b}. It describes the motion of 
$N$, $N \geq 2$, point masses, with positions $q_{1}, \ldots, q_{N}$, momenta $p_{1}, \ldots, p_{N}$
and masses $m_{1}, \ldots, m_{N}$, moving up and down a vertical line and colliding elastically with 
each other. The bottom particle collides elastically with a
rigid floor placed at position $q_{1} = 0$. For convenience, we will refer to 
the point particles as balls
The system is an autonomous Hamiltonian system, with  Hamiltonian given by the sum 
of the kinetic and linear potential energy of each ball. It possesses a smooth invariant measure with 
respect to the Hamiltonian flow and with respect to a
suitable Poincar\'e map $T$, describing the movement of the balls 
from one collision to the next. We denote the underlying Poincar\'e section for this map by 
$\mathcal{M}^+$ and its invariant measure by $\mu$. 
One aspect that makes the description of the dynamics 
cumbersome is the presence of singularities. These are codimension one manifolds in the phase space, 
on which the dynamics are not well-defined, in particular it has two different images. 
A point belongs to the singularity manifold, if its next collision is either between three balls or 
two balls with the floor. 

Dynamicists first tried to answer the question whether the system of $N$, $N \geq 2$, falling balls has 
$2N-2$ non-zero Lyapunov exponents on a positive measure set of the phase space. The exceptional two 
directions with a zero exponent are the direction of the flow and the directions transversal to the 
energy surface. 
Wojtkowski was able to prove, 
that two and three falling balls have non-zero Lyapunov exponents almost everywhere \cite{W90a}. 
He supplemented this result by proving that an arbitrary number of balls exposed to a certain 
family of non-linear potential fields have non-zero Lyapunov exponents almost everywhere \cite{W90b}. 
The most general result, regarding the linear potential field, 
is due to Sim\'anyi: For $N$, $N \geq 2$, falling balls, $\mu$-a.e. point $x 
\in \mathcal{M}^+$ has non-zero Lyapunov exponents \cite{S96}.    
In \cite{W98} Wojtkowski found an elegant way of proving the existence of non-zero Lyapunov 
exponents for a large class of falling balls systems. He first considers balls falling next to each 
other on a moving floor. By applying concrete stacking rules it is possible to obtain a 
variety of falling ball systems, such as the original one introduced in \cite{W90a} as a special case. 
The study of hyperbolicity is carried out by equivalently looking at the system of a particle falling 
in a wedge.

The underlying motivation of this work is to contribute to the solution of the long time open problem 
of ergodicity for three or more balls. For two balls, the system is already known to be ergodic 
\cite[p. 70 -72]{LW92}, provided $m_1 > m_2$. 
Since the system of three falling balls has non-zero Lyapunov exponents everywhere, the theory of 
Katok-Strelcyn \cite{KS86} yields, that the phase space partitions into at most countably many 
components on which the conditional smooth measure is ergodic. 
A reliable method to check the ergodicity of such systems is the local ergodic theorem 
\cite{CS87,KSSz90,LW92}. 
In the present work we will follow the local erogdic theorem version of Liverani and Wojtkoswki 
\cite{LW92}. 
For its application, the local ergodic theorem needs the following five conditions to hold, namely,
\begin{center}
\begin{enumerate}
 \item Chernov-Sinai ansatz,\\[-0.2cm]
 \item Non-contraction property,\\[-0.2cm]
 \item Continuity of Lagrangian subspaces,\\[-0.2cm]
 \item Regularity of singularity sets,\\[-0.2cm]
 \item Proper Alignment. 
\end{enumerate}
\end{center}
The validity of these conditions guarantees the existence of an open neighbourhood, around a point 
with non-vanishing Lyapunov exponents, that lies (mod 0) in one ergodic component. 
To prove, that there is only one ergodic component needs the validity 
of a transitivity argument. Namely, the set of points with a sufficient amount of expansion 
must have full measure and be arcwise connected. 
We will refer to this property as the abundance of least expanding points.
If the latter is true, one can build a chain of the aforementioned 
open neighbourhoods from any point with sufficient (or least) expansion to another. 
These neighbourhoods intersect pairwise on a subset of positive measure and, hence, there 
can only be one ergodic component. For three or more balls only condition 3 is known \cite{LW92} 
to be true. 

In their approach to ergodicity, Liverani and Wojtkowski introduced \cite{LW92} the property of 
(strict) unboundedness for a sequence of derivatives $(d_{T^nx}T)_{n \in \mathbb{N}}$. It roughly says, that 
the expansion (measured with respect to a special quadratic form) of any vector from the contracting 
cone field goes to infinity. In their terminology, it follows immediately that if 
$(d_{T^nx}T)_{n \in \mathbb{N}}$ is strictly unbounded everywhere then the Chernov-Sinai ansatz 
holds. Additionally, the abundance of least expanding points follows as a simple corollary.

The proof of the strict unboundedness property for every phase point is the main task of this work 
(see Section 2 for more details). 
For this, we will partially use techniques introduced in \cite{W98}, which allow us to study the
system of falling balls as a particle falling in a wedge. The results obtained 
from the latter analysis will be used to slightly modify the approach to strict unboundedness 
in \cite{LW92} for our needs.

Another important issue, which we clarify in a separate subsection is the state of the proper 
alignment condition (see Subsection \ref{state}). 
By some experts it has been wrongly assumed not to hold. We will thoroughly explain that 
this condition can still be verified and is, thus, an open problem. Further, we 
will use the strict unboundedness property to analyze in Subsection \ref{iterates}, how the 
set of not properly aligned points behaves under sufficiently large iterates.
We point out that for a 
specific mass ratio the configuration space of the falling balls systems can be unfolded to 
a billiard table where the proper alignment condition holds (see Subsection \ref{special}). 
The latter was discovered by Wojtkowski \cite{W16}.

On the same subject, Chernov formulated \cite{Ch93}, in the realm of semi-dispersing billiards, a 
transversality condition, which can serve as a substitute for the proper alignment condition. 
We will show, that in the framework of symplectic maps, Chernov's 
transversality condition is actually equivalent to the proper alignment condition (see 
Lemma \ref{equivalence}).

The paper is organized in the following way:

In Section 2 we briefly summarize the main results of this paper, which are the strict unboundedness for 
every orbit, the Chernov-Sinai ansatz and the abundance of least expanding points. It will also 
be shown, that the latter two results follow at once from the strict unboundedness property of 
every orbit.

In Section 3 we introduce the system of three falling balls.

In Section 4 we recall the standard method for studying Lyapunov exponents 
in Hamiltonian systems \cite{W91} and recall what has been done for the system of falling balls so far.

In Section 5 we explain the matter of ergodicity. It contains a detailed discussion of 
the local ergodic theorem, the proper alignment condition, Chernov's 
transversality condition and the abundance of least expanding points.

In Section 6 we begin with the first part of the proof of the strict unboundedness property. 
This section is completely written in the language of Liverani and Wojtkowski \cite{LW92} and 
explains how we use our new results in order to modify their proof of the unboundedness property.

In Section 7 we introduce the system of a particle falling in a three dimensional wedge 
from \cite{W98}. 
Its necessity stems from the fact, that for a special type of wedges this system is equivalent to the 
system of falling balls with particular masses. In the last subsection we will explain that the 
proper alignment condition is valid in these special wedges.

In Section 8 we utilize the results of Section 6 and 7 to complete the proof of the strict 
unboundedness property.

\section{Main results}
Denote by $\mathcal{M}^+$ the phase space, which is partitioned ($\operatorname{mod} 0$) into subsets 
$\mathcal{M}_i^+$, $i = 1, 2, 3$, where 
each subset describes the moment right after collision of balls $i-1$ and $i$. For $i - 1 = 0$, we have a collision with the floor, i.e. 
$q_1 = 0$. Let $T:\ \mathcal{M}^+ \circlearrowleft$ be the Poincar\'e map, describing the movement from one collision to the next. 
After applying Wojtkowski's convenient coordinate transformation $(q,p) \to (h,v) \to (\xi, \eta)$ (see \cite{W90a}), we get a contracting cone field
\begin{align}
\mathcal{C}(x) &= \{(\delta \xi, \delta \eta) \in \mathbb{R}^{3} \times \mathbb{R}^{3}:\ 
Q(\delta \xi, \delta \eta) > 0,\ \delta \xi_1 = 0,\ \delta \eta_1 = 0\} \cup \{\vec{0}\},\nonumber 
\end{align}
where $(\delta \xi, \delta \eta)$ denote the coordinates in tangent space. The cone field is defined by the quadratic form 
\begin{align}
Q(\delta \xi, \delta \eta) = \sum_{i = 1}^3 \delta \xi_i \delta \eta_i.\nonumber
\end{align}
Denote by $\overline{\mathcal{C}(x)}$ the closure of the cone $\mathcal{C}(x)$. The sequence $(d_{T^n x}T)_{n \in \mathbb{N}}$ is 
called unbounded, if
\begin{align}
\lim_{n \to +\infty} Q(d_{x}T^{n}v) = +\infty,\ \forall\ v \in \mathcal{C}(x) \setminus \{\vec{0}\}.\nonumber
\end{align}
and strictly unbounded, if
\begin{align}
\lim_{n \to +\infty} Q(d_{x}T^{n}v) = +\infty,\ \forall\ v \in \overline{\mathcal{C}(x)} \setminus \{\vec{0}\}.\nonumber
\end{align}

\begin{main}\label{main}
	For every $x \in \mathcal{M}^+$, we have 
	\begin{align}
	\lim_{n \to + \infty}Q(d_xT^{\pm n}(\delta \xi, \delta \eta)) = \pm \infty,\nonumber
	\end{align}
	for all $(\delta \xi, \delta \eta) \in \overline{\mathcal{C}(x)} \setminus \{\vec{0}\}$.
\end{main}
We will formulate the proof of the Main Theorem only for the positive orbit $(d_{T^n x}T)_{n \in \mathbb{
N}}$, since the proof for $(d_{T^n x}T)_{n \in \mathbb{Z}^-}$ is exactly the same.

The singularity manifold on which $T$ resp. $T^{-1}$ is not well-defined is given by $\mathcal{S}^+$ 
resp. $\mathcal{S}^-$. Let $\mu_{|_{\mathcal{S}^+}}$ resp. $\mu_{|_{\mathcal{S}^-}}$ be the 
restriction to $\mathcal{S}^+$ resp. $\mathcal{S}^-$ of the smooth $T$-invariant measure $\mu$.

The validity of the Main Theorem immediately establishes the Chernov-Sinai ansatz, which is one of the 
conditions of the Local Ergodic Theorem.
\begin{cs}
	For $\mu_{|_{\mathcal{S}^{\pm}}}$ -a.e. $x \in \mathcal{S}^{\pm}$, we have
	\begin{align}
	\lim_{n \to +\infty} Q(d_xT^{\mp n}(\delta \xi, \delta \eta)) = \mp\infty,\nonumber
	\end{align}
	for all $(\delta \xi, \delta \eta) \in \overline{\mathcal{C}(x)} \setminus \{\vec{0}\}$.
\end{cs}
The least expansion coefficient $\sigma$, for $n \geq 1$ and $x \in \mathcal{M}^+$, is defined as
\begin{align}
\sigma(d_{x}T^n) = \inf_{v \in \mathcal{C}(x)} \sqrt{\frac{Q(d_{x}T^nv)}{Q(v)}}.\nonumber
\end{align}
A point $x \in \mathcal{M}^+$, is called least expanding, if there 
exists $n = n(x) \geq 1$, such that $\sigma(d_{x}T^n) > 1$.

The last result is the abundance of least expanding points. 
It can be described as a transitivity argument, which acts in specifying the size of 
the ergodicity domain in phase space by connecting open neighbourhoods, which lie $(\operatorname{mod} 0)$ in one ergodic component. 
\begin{abundance}
The set of least expanding points has full measure and is arcwise connected.
\end{abundance}
As the Chernov-Sinai ansatz, the abundance of least expanding points follow at once from the 
Main Theorem, since $(d_{T^nx}T)_{n\in\mathbb{N}}$ is strictly unbounded if and only if 
$\lim_{n \to \infty} \sigma(d_xT^n) = \infty$ (see \cite[Theorem 6.8]{LW92}).

\section{The system of three falling balls}
Let $q_{i} = q_{i}(t)$ be the position, $p_{i} = p_{i}(t)$ the momentum and $v_i = v_i(t)$ the velocity 
of the $i$-th ball. The balls are aligned on top of each other and are therefore confined to 
\begin{align}
\mathcal{N}(q,p) = \{(q,p) \in \mathbb{R}^3 \times \mathbb{R}^3:\ 0 \leq q_1 \leq q_2 \leq q_3\}.
\nonumber
\end{align} 
The momenta and the velocities are related by $p_i = m_iv_i$. We 
assume that the masses $m_{i}$ satisfy $m_{1} > m_{2} \geq m_{3}$. The movements of the balls are a 
result of a linear potential field and their kinetic energies. The total energy of the system is given 
by the Hamiltonian function
\begin{align}
H(q,p) = \sum_{i = 1}^{3} \frac{p_{i}^{2}}{2m_{i}} + m_{i}q_{i}.\nonumber
\end{align}

The Hamiltonian equations are
\begin{align}\label{equations}
\begin{array}{ccc}
\dot{q_{i}} & = & \dfrac{p_{i}}{m_{i}},\\[0.3cm]
\dot{p_{i}} & = & -m_{i}.
\end{array}
\end{align}

The dots indicate differentiation with respect to time $t$ and the Hamiltonian vector field on the 
right hand side will be denoted as $X_{H}(q,p)$. The solutions to these equations are
\begin{align}\label{equation2}
\begin{array}{ccl}
q_{i}(t) & = & -\dfrac{t^{2}}{2} + t\dfrac{p_{i}(0)}{m_{i}} + q_{i}(0),\\[0.3cm]
p_{i}(t) & = & -tm_{i} + p_{i}(0),
\end{array}
\end{align}
which form parabolas in $(t, q_{i}(t)) \subset \mathbb{R} \times \mathbb{R}_{+}$. It is clear from the 
choice of the linear potential field, that the acceleration of each ball points downwards and, thus, 
these parabolas cannot escape to infinity. Hence, for every initial condition $(q,p)$ the balls go 
through every collision in finite time and, thus, every collision happens infinitely often. The energy 
manifold $E_c$ and its tangent space $\mathcal{T}E_c$ are given by 
\begin{align} 
E_c &= \{(q,p) \in \mathbb{R}_{+}^{3} \times \mathbb{R}^{3}:\ H(q,p) =  \sum_{i = 1}^{3} 
\frac{p_{i}^{2}}{2m_{i}} + m_{i}q_{i} = c\},\nonumber\\
\mathcal{T}_{(q,p)}E_c &= \{(\delta q, \delta p) \in \mathbb{R}^{3} \times \mathbb{R}^{3}:\   
d_{(q,p)}H(\delta q, \delta p) 
= \sum_{i=1}^{3}\dfrac{p_{i}\delta p_{i}}{m_{i}} + m_{i}\delta q_{i} = 0\}.\nonumber
\end{align}
Including the restriction of the balls positions amounts to $E_c \cap \mathcal{N}(q,p)$.

The Hamiltonian vector field (\ref{equations}) gives rise to the Hamiltonian flow 
\begin{align}
 \phi:\ &\mathbb{R} \times E_c \cap \mathcal{N}(q,p) \to E_c \cap \mathcal{N}(\phi(t,(q,p)), \nonumber\\
 &(t,(q,p)) \mapsto \phi(t,(q,p)).\nonumber 
\end{align}
For convenience, the image will also be written with the time variable as superscript, i.e. 
$\phi(t,(q,p)) = \phi^{t}(q,p)$.\\
The standard symplectic form $\omega = \sum_{i = 1}^3 dq_i \wedge dp_i$ induces the symplectic volume element 
$\Omega = \bigwedge_{i = 1}^3 dq_i \wedge dp_i$. We restrict it to $\iota (u) \Omega$, 
where $u$ is any vector satisfying $dH(u)=1$ and $\iota$ is the interior derivative. Since the flow preserves 
the standard symplectic form, it preserves the volume element and, hence, the Liouville measure $\nu$ on $E_c \cap \mathcal{N}(q,p)$ obtained from it.
We define the Poincar\'e section, which describes the states right after a collision as
$\mathcal{M}^{+} = \mathcal{M}_{1}^{+} \cup \mathcal{M}_{2}^{+} \cup \mathcal{M}_{3}^{+}$, with
\begin{align}
&\mathcal{M}_{1}^{+} := \{(q,p) \in E_c \cap \mathcal{N}(q,p): q_{1} = 0,\ p_{1}/m_{1} \geq 0\},
\nonumber\\
&\mathcal{M}_{i}^{+} := \{(q,p) \in E_c \cap \mathcal{N}(q,p): q_{i-1} = q_{i},\ 
p_{i-1}/m_{i-1} \leq p_{i}/m_{i}\},\ i = 2,3.\nonumber
\end{align}

In the same way we define the set of states right before collision $\mathcal{M}^{-} = 
\mathcal{M}_{1}^{-} \cup \mathcal{M}_{2}^{-} \cup \mathcal{M}_{3}^{-}$, by
\begin{align}
&\mathcal{M}_{1}^{-} := \{(q,p) \in E_c \cap \mathcal{N}(q,p): q_{1} = 0,\ p_{1}/m_{1} < 0\},\nonumber\\
&\mathcal{M}_{i}^{-} := \{(q,p) \in E_c \cap \mathcal{N}(q,p): q_{i-1} = q_{i},\ 
p_{i-1}/m_{i-1} > p_{i}/m_{i}\},\ i = 2,3.\nonumber
\end{align}

The '+' resp. '-' superscript refer to the states right after resp. before collision. 
The system of falling balls is considered as a hard ball system with fully elastic collisions. During a 
collision of the balls $i$ and $i+1$ the momenta resp. velocities change according to
\begin{align}\label{collisionqp}
\begin{array}{ccc}
p_{i}^{+} & = & \gamma_{i} p_{i}^{-} + (1 + \gamma_{i})p_{i+1}^{-},\\
p_{i+1}^{+} & = & (1 - \gamma_{i})p_{i}^{-} - \gamma_{i} p_{i+1}^{-},\\[0.1cm]
v_{i}^{+} & = & \gamma_{i} v_{i}^{-} + (1 - \gamma_{i})v_{i+1}^{-},\\
v_{i+1}^{+} & = & (1 + \gamma_{i})v_{i}^{-} - \gamma_{i} v_{i+1}^{-},
\end{array}
\end{align}
where $\gamma_{i} = (m_{i} - m_{i+1})/(m_{i} + m_{i+1})$, $i = 1, 2$, and when the bottom particle 
collides with the floor the sign of its momentum is simply reversed
\begin{align}\label{v1p1}
\begin{array}{ccc}
p_{1}^{+} = -p_{1}^{-},\\
v_1^+ = -v_1^-.
\end{array}
\end{align}

These collision laws are described by the linear collision map 
\begin{align}%\label{cmap}
\begin{array}{rl}
 \Phi_{i-1,i}:\ &\mathcal{M}^{-} \to \mathcal{M}^{+},\nonumber\\[0.2cm]
 &(q,p^{-}) \mapsto (q,p^{+}).\nonumber
\end{array}
\end{align}

We will write $\Phi$ if we do not want to refer to any specific collision. 
Let $\tau: M \to \mathbb{R}_{+}$ be the first return time to $\mathcal{M}^{-}$. 
We define the Poincar\'e map as 
\begin{align}
T:\  &\mathcal{M}^{+} \to \mathcal{M}^{+},\nonumber \\
&(q,p) \mapsto \Phi \circ \phi^{\tau(q,p)}(q,p).\nonumber
\end{align}

$T$ is the collision map, that maps from one collision to the next. By restricting the volume form on 
$E_c \cap \mathcal{N}(q,p)$ 
with respect to the direction of the flow we obtain the volume form $\iota(X_H) \iota (u) \Omega$. 
This volume form defines a smooth measure $\mu$ on $\mathcal{M}^+$, which is $T$-invariant.
Our dynamical system can be stated as the triple $(\mathcal{M}^{+}, T, \mu)$. 
Each $\mathcal{M}_{i}^{+}$ and $\mathcal{M}_{i}^{-}$ further partitions ($\operatorname{mod} 0$) into
\begin{align}
&\mathcal{M}_{i,j}^{+} = \{x \in \mathcal{M}_{i}^{+}:\ Tx \in \mathcal{M}_{j}^{+}\},\ 
j \in \{1,2,3\},\ j \neq i,\nonumber\\
&\mathcal{M}_{i,j}^{-} = \{x \in \mathcal{M}_{i}^{-}:\ 
\Phi_{j-1,j} \circ T^{-1} \circ \Phi_{i-1,i} x \in \mathcal{M}_{j}^{-}\},\ 
j \in \{1,2,3\},\ j \neq i.\nonumber
\end{align}

It can be calculated, that $\mu(\mathcal{M}_{i,j}^{\pm}) > 0$.   
The system of falling balls possesses codimension one singularity manifolds
\begin{align}
 &\mathcal{S}_{1,2}^{+} = \{(q,p) \in \mathcal{M}_1^+:\ \phi^{\tau(q,p)}(q,p) \in 
 \mathcal{M}_2^- \cap \mathcal{M}_3^-\},\nonumber\\
 &\mathcal{S}_{1,2}^{-} = \{(q,p) \in \mathcal{M}_2^+ \cap \mathcal{M}_3^+:\ 
 T(q,p) \in  \mathcal{M}_1^+\},\nonumber\\
 &\mathcal{S}_{3,1}^{+} = \{(q,p) \in \mathcal{M}_3^+:\ \phi^{\tau(q,p)}(q,p) \in 
 \mathcal{M}_1^- \cap \mathcal{M}_2^-\},\nonumber\\
 &\mathcal{S}_{3,1}^{-} = \{(q,p) \in \mathcal{M}_1^+ \cap \mathcal{M}_2^+:\ 
 T(q,p) \in  \mathcal{M}_3^+\}.\nonumber
\end{align}

The states in $\mathcal{S}_{1,2}^{+}, \mathcal{S}_{1,2}^{-}$ face a triple collision next, while the
states in $\mathcal{S}_{3,1}^{+}, \mathcal{S}_{3,1}^{-}$ experience a collision of the lower two balls 
with the floor next. 
The maps $T$ resp. $T^{-1}$ are not well-defined on the sets 
$\mathcal{S}_{1,2}^{+}, \mathcal{S}_{3,1}^{+}$ resp. $\mathcal{S}_{1,2}^{-}, \mathcal{S}_{3,1}^{-}$ , 
because they have two different images. This happens because we can approximate,
say, a triple collision in two ways: Once by letting ball one and two collide an instant before ball 
two and three and vice versa. When the trajectory hits a singularity, we will continue the system 
on both branches separately. In this way, the results obtained in this work hold for every point. 

We abbreviate 
\begin{align}
&\mathcal{S}^{\pm} = \mathcal{S}_{1,2}^{\pm} \cup \mathcal{S}_{3,1}^{\pm},\nonumber\\
&\mathcal{S}_n^{\pm} = \mathcal{S}^{\pm} \cup T^{\mp 1}\mathcal{S}^{\pm} \cup \ldots 
\cup T^{\mp (n-1)}\mathcal{S}^{\pm}.\nonumber
\end{align}

\section{Lyapunov exponents}
We subject our system to two well-discussed coordinate transformations $(q,p) \to (h,v) \to (\xi, \eta)$ 
introduced in \cite{W90a}. The first one is given by 
\begin{align}
h_{i} &= \dfrac{p_{i}^{2}}{2m_{i}} + m_{i}q_{i},\nonumber\\
v_{i} &= \dfrac{p_{i}}{m_{i}},\nonumber
\end{align}
while the second one is a linear coordinate transformation
\begin{align}
\xi_i &= A^{-1}h_i,\nonumber\\
\eta_i &= A^{T}v_i,\nonumber
\end{align}
where $A$ is a symplectic matrix depending only on the masses $m_i$ \cite[p. 520]{W90a}. The energy 
manifold and its tangent space take the form
\begin{align}
E_c &= \{(\xi, \eta) \in \mathbb{R}^{3} \times \mathbb{R}^{3}:\ H(\xi,\eta) = \xi_{1} = c\},\nonumber\\
\mathcal{T}E_c &= \{(\delta \xi, \delta \eta) \in \mathbb{R}^{3} \times \mathbb{R}^{3}:\  dH(\xi,\eta) = 
\delta \xi_{1} = 0\}.\nonumber
\end{align}

The Hamiltonian vector field $X_{H}(\xi,\eta) = (0,0,0,-1,0,0)$ becomes constant. 
In these coordinates, the derivative of the flow $d\phi^t$ equals the identity map. 
Thus, only the derivatives of the collision maps $d\Phi_{i-1,i}$ are relevant to the 
dynamics in tangent space. In these coordinates the collision maps are given by 
\begin{align}
d\Phi_{0,1} = \nonumber
\begin{pmatrix}
\operatorname{id}_3 & 0\\
B & \operatorname{id}_3
\end{pmatrix},\
d\Phi_{1,2} = \nonumber
 \begin{pmatrix}
M_1 & U_1\\
0 & M_1^T
\end{pmatrix},\
d\Phi_{2,3} = \nonumber
\begin{pmatrix}
M_2 & U_2\\
0 & M_2^T
\end{pmatrix}.
\end{align}
where 
\begin{align}
B &= 
 \begin{pmatrix}
  1 & 0 & 0\\
  0 & \beta & 0\\
  0 & 0 & 0
 \end{pmatrix},\
U_1 = 
 \begin{pmatrix}
  0 & 0 & 0\\
  0 & \alpha_1 & 0\\
  0 & 0 & 0
 \end{pmatrix},\ 
U_2 = 
 \begin{pmatrix}
  0 & 0 & 0\\
  0 & 0 & 0\\
  0 & 0 & \alpha_2
 \end{pmatrix},\nonumber\\
\operatorname{id}_3 &=
  \begin{pmatrix}
 1 & 0 & 0\\
 0 & 1 & 0\\
 0 & 0 & 1
 \end{pmatrix},\
 M_1 =
  \begin{pmatrix}
 1 & 0 & 0\\
 0 & -1 & 1+\gamma_{1}\\
 0 & 0 & 1
 \end{pmatrix},\
 M_2 =
  \begin{pmatrix}
 1 & 0 & 0\\
 0 & 1 & 0\\
 0 & 1 - \gamma_{2} & -1
 \end{pmatrix}.\nonumber
\end{align}

The terms in the matrices are given by
\begin{align}\label{alpha}
\beta = -\dfrac{2}{m_{1}v_{1}^{-}},\quad \alpha_{i} = 
\dfrac{2m_{i}m_{i+1}(m_{i} - m_{i+1})(v_{i}^{-} - v_{i+1}^{-})}{(m_{i} + m_{i+1})^{2}}.
\end{align}

A Lagrangian subspace $V$ is a linear space of maximal dimension on which the symplectic form vanishes. 
In general, every vector $v \in \mathbb{R}^6$ can be uniquely decomposed by a pair of two 
given transversal Lagrangrian subspaces $(V_1$, $V_2)$, i.e. $v = v_1 + v_2$, $v_{i} \in V_{i}$, 
$i = 1,2$. For a pair of transversal Lagrangian subspaces $(V_1, V_2)$ we can define 
a quadratic form $Q$ by
\begin{align}
Q:\ &\mathbb{R}^{6} \to \mathbb{R}\nonumber\\
&v \mapsto Q(v) = \omega(v_{1},v_{2})\nonumber
\end{align}

The canonical pair of transversal Lagrangian subspaces in $\mathbb{R}^{6}$ is given by
\begin{align}
W_{1} = \{(\delta \xi, \delta \eta) \in \mathbb{R}^{3} \times \mathbb{R}^{3}:\ 
\delta \eta_{1} = \delta \eta_{2} = \delta \eta_{3} = 0\},\nonumber\\
W_{2} = \{(\delta \xi, \delta \eta) \in \mathbb{R}^{3} \times \mathbb{R}^{3}:\ 
\delta \xi_{1} = \delta \xi_{2} = \delta \xi_{3} = 0\}.\nonumber 
\end{align}

Restricting both to $\mathcal{T}E$ and excluding the direction of the flow gives
\begin{align}
\begin{array}{rl}\label{constant}
&L_{1} = \{(\delta \xi, \delta \eta) \in \mathbb{R}^{3} \times \mathbb{R}^{3}:\ 
\delta \xi_{1} = 0,\ \delta \eta_{i} = 0,\ i=1,2,3\},\\[0.2cm]
&L_{2} = \{(\delta \xi, \delta \eta) \in \mathbb{R}^{3} \times \mathbb{R}^{3}:\ 
\delta \eta_{1} = 0,\ \delta \xi_{i} = 0,\ i=1,2,3\}.
\end{array}
\end{align}

For the pair $(L_1, L_2)$, the quadratic form $Q$ becomes the Euclidean inner product
\begin{align}
 Q(\delta \xi, \delta \eta) = \langle \delta \xi, \delta \eta \rangle.\nonumber
\end{align}

We see immediately that $Q(L_{i}) = 0$. Also, $Q$ is continuous and homogeneous of degree two.
Using the quadratic form $Q$ we can define the open cones
\begin{align}
\mathcal{C}(x) &= \{(\delta \xi, \delta \eta) \in L_1 \oplus L_2:\ 
Q(\delta \xi, \delta \eta) > 0\} \cup \{\vec{0}\},\nonumber\\
\mathcal{C}^{\prime}(x) &= \{(\delta \xi, \delta \eta) \in L_1 \oplus L_2:\ 
Q(\delta \xi, \delta \eta) < 0\} \cup \{\vec{0}\}.\nonumber
\end{align}

Denote by $\overline{\mathcal{C}(x)}$ the closure of the cone $\mathcal{C}(x)$.
\begin{definition}\label{Defq}
1. The cone field $\{\mathcal{C}(x),\ x \in \mathcal{M}^{+}\}$, is called 
 invariant for $x \in \mathcal{M}^+$, if 
 \begin{align}
 d_xT \overline{\mathcal{C}(x)} \subseteq \overline{\mathcal{C}(Tx)},\nonumber
 \end{align}
2. The cone field $\{\mathcal{C}(x),\ x \in \mathcal{M}^{+}\}$, is called 
 eventually strictly invariant for $x \in \mathcal{M}^+$, if there exists
a $k \geq 1$, such that 
\begin{align}
 d_xT^k \overline{\mathcal{C}(x)} \subset \mathcal{C}(T^kx).\nonumber
\end{align}
3. The monodromy map $d_xT$ is called Q-monotone for $x \in \mathcal{M}^{+}$, if 
\begin{align}
Q(d_xT(\delta \xi, \delta \eta)) \geq Q(\delta \xi, \delta \eta), \nonumber
\end{align}
for all $(\delta \xi, \delta \eta) \in L_1 \oplus L_2$.\\
4. The monodromy map $d_xT$ is called eventually strictly Q-monotone for $x \in \mathcal{M}^{+}$, 
if there exists a $k \geq 1$, such that 
\begin{align}
Q(d_xT^k(\delta \xi, \delta \eta)) > Q(\delta \xi, \delta \eta), \nonumber
\end{align}
for all $(\delta \xi, \delta \eta) \in L_1 \oplus L_2 \setminus \{\vec{0}\}$.
\end{definition}

Statement \textit{1.} resp. \textit{2.} is equivalent to statement \textit{3.} resp. \textit{4.} 
(see e.g. \cite[Theorem 4.1]{LW92}). The 
following lemma establishes eventual strict Q-monotonicity by using only the 
evolution of the Lagrangian subspaces $L_1$ and $L_2$
(see e.g. \cite[Lemma 2]{W90a}).
\begin{lemma}\label{lem}
The monodromy map $d_xT$ is eventually strictly Q-monotone for 
$x \in \mathcal{M}^{+}$, if there exists $N \geq 1$, such that for all 
$(\delta \xi, 0) \in L_1$ and $(0, \delta \eta) \in L_2$,
\begin{align}
Q(d_xT^N (\delta \xi, 0)) > 0\ \text{ and }\ Q(d_xT^N (0, \delta \eta)) > 0.\nonumber
\end{align}
\end{lemma}

In order to get nonzero Lyapunov exponents Wojtkowski introduced \cite[p. 516]{W90a} a 
criterion, which links eventual strict Q-monotonicity to nonuniform hyperbolic 
behaviour
\begin{qcrit}
If $d_xT$ is eventually strictly Q-monotone for $\mu$-a.e. $x \in \mathcal{M}^{+}$, then 
all Lyapunov exponents, except for two\footnote{The exceptional directions with zero Lyapunov exponents 
are the direction of the flow and the ones contained in the subset $\{v:\ dH(v) \neq 0\}$.}, are non-zero.
\end{qcrit}

For $N$, $N \geq 2$, balls, Wojtkowski proved \cite{W90a}, that $d_xT$ is Q-monotone for every point in 
$\mathcal{M}^+$. Wojtkowski strengthened this statement in the case of three balls with upward decreasing masses, 
by proving eventual strict Q-monotonicity for every
\footnote{Even though Proposition 3 in \cite{W90a} is stated for almost every point 
$x \in \mathcal{M}^+$, the reader will discover, when carefully reading the proof, that it 
actually holds for every $x \in \mathcal{M}^+$.} point in $\mathcal{M}^+$ 
\cite[Proposition 3]{W90a}. 
Afterwards Sim\'anyi proved \cite{S96}, that $d_xT$ is eventually strictly Q-monotone for 
$\mu$-a.e. $x \in \mathcal{M}^+$ and an arbitrary number of balls.\\ 
We close this subsection by formulating the (strict) unboundedness property and the least expansion coefficient, which will be used to establish criteria for ergodicity.\\ 
The least expansion coefficient $\sigma$, for $n \geq 1$ and $x \in \mathcal{M}^+$, is defined as
\begin{align}\label{leastexp}
\sigma(d_{x}T^n) = \inf_{v \in \mathcal{C}(x)} \sqrt{\frac{Q(d_{x}T^nv)}{Q(v)}}.
\end{align}

\begin{definition}\label{unboundedness}
1. The sequence $(d_{T^n x}T)_{n \in \mathbb{N}}$ is called unbounded, if
\begin{align}
\lim_{n \to +\infty} Q(d_{x}T^{n}v) = +\infty,\ \forall\ v \in \mathcal{C}(x) \setminus \{\vec{0}\}.\nonumber
\end{align}

2. The sequence $(d_{T^n x}T)_{n \in \mathbb{N}}$ is called strictly unbounded, if
\begin{align}
\lim_{n \to +\infty} Q(d_{x}T^{n}v) = +\infty,\ \forall\ v \in \overline{\mathcal{C}(x)} 
\setminus \{\vec{0}\}.\nonumber
\end{align}
\end{definition}
The least expansion coefficient and the property of strict unboundedness relate to each 
other in the following way
\begin{theorem}[Theorem 6.8, \cite{LW92}]\label{relation}
The sequence $(d_{T^nx}T)_{n\in\mathbb{N}}$ is strictly unbounded if and only if 
$\lim_{n \to \infty} \sigma(d_xT^n) = \infty$.
\end{theorem}

\section{Ergodicity}
The theory of Katok-Strelcyn \cite{KS86} implies, that since our system has non-zero Lyapunov 
exponents almost everywhere, we can partition the phase space $\mathcal{M}^+$ into countably many 
components on which the conditional smooth measure is ergodic. To prove that there is only one 
ergodic component the following two points need to be verified
\begin{enumerate}
\item Local Ergodicity. \label{geh}
\item Abundance of least expanding points. \label{scheissn}
\end{enumerate}
\subsection{Local Ergodicity}
We start with the following
\begin{definition}\label{reg}
A compact subset $X \subset \mathcal{M}^{+}$, is called regular if
\begin{enumerate}
	\item $X = \bigcup_{i = 1}^{n} I_{i}$, where $I_{i}$ are compact subsets,
	\item $\operatorname{dim} I_{i} = 3$,
	\item $I_{i} \cap I_{j} \subset \partial I_{i} \cup \partial I_{j},\ i \neq j$,
	\item $\partial I_{i} = \bigcup_{j = 1}^{m} H_{i,j}$, where $\operatorname{dim} H_{i,j} = 2$ and $H_{i,j}$ is compact.
\end{enumerate}
\end{definition}

Local ergodicity amounts to showing that around a point with least expansion 
coefficient larger than three, it is possible to find an open neighbourhood, which lies (mod 0) in one 
ergodic component. To claim this, one needs to check the following five conditions
\begin{con}[Regularity of singularity sets]\label{condi3}
	The singularity sets $\mathcal{S}_n^{+}$ and $\mathcal{S}_n^{-}$ are both regular
	sets for every $n \geq 1$. 
\end{con}

\begin{con}[Non-contraction property]\label{condi1}
There exists $\zeta > 0$, such that for every $n \geq 1$, 
$x \in \mathcal{M}^+ \setminus \mathcal{S}_n^+$,  and 
$(\delta \xi, \delta \eta) \in \overline{\mathcal{C}(x)}$, we have 
\begin{align}
\|d_xT^n(\delta \xi, \delta \eta)\| \geq \zeta \|(\delta \xi, \delta \eta)\|.\nonumber
\end{align}
\end{con}

\begin{con}[Chernov-Sinai Ansatz]\label{condi2}
For $\mu_{|_{\mathcal{S}^{\pm}}}$ -a.e. $x \in \mathcal{S}^{\pm}$, we have
\begin{align}
\lim_{n \to +\infty} Q(d_xT^{\mp n}(\delta \xi, \delta \eta)) = \mp\infty,\nonumber
\end{align}
for all $(\delta \xi, \delta \eta) \in \overline{\mathcal{C}(x)}$.
\end{con}

\begin{con}[Continuity of Lagrangian subspaces]\label{condi4}
The ordered pair of transversal Lagrangian subspaces $(L_{1}(x)$, $L_{2}(x))$ varies continuously in 
$\operatorname{int}\mathcal{M}^{+}$. 
\end{con}

\begin{con}[Proper Alignment]\label{condi5}
There exists $N \geq 0$, such that for every $x \in \mathcal{S}^{+}$ 
resp. $\mathcal{S}^{-}$, we have $d_xT^{-N} v_{x}^{+}$ resp. $d_xT^N v_{x}^{-}$ belong to 
$\overline{\mathcal{C}^\prime (T^{-N} x)}$ resp. $\overline{\mathcal{C}(T^N x)}$, 
where $v_{x}^{+}$ resp. $v_{x}^{-}$ are the characteristic lines\footnote{The characteristic line 
$v_{x}^{\pm}$ is a vector of $\mathcal{T}_{x}\mathcal{S}^{\pm}$ that has the property of 
annihilating every other vector $w \in \mathcal{T}_{x}\mathcal{S}^{\pm}$ with respect to the 
symplectic form $\omega$, i.e. $\omega(v_{x}^{\pm},w) = 0$, $\forall\ w \in \mathcal{T}_{x}\mathcal{S}^
{\pm}$. Alternatively stated, it is the $\omega$-orthogonal complement of $\mathcal{T}_{x}\mathcal{S}^{
\pm}$. Note, that in symplectic geometry the $\omega$-orthogonal complement of a codimension one 
subspace is one dimensional.} 
of $\mathcal{T}_x\mathcal{S}^{+}$ resp. $\mathcal{T}_x\mathcal{S}^{-}$.
\end{con}
At the moment, for three or more falling balls, only Condition \ref{condi4} has been verified.
This is in fact easy to see, because the canonical pair of transversal Lagrangian subspaces 
(\ref{constant}) does not depend on the base point $x$ and is therefore constant in $\mathcal{M}^+$. 

\begin{local}
If Conditions \ref{condi3} - \ref{condi5} are satisfied, then for any $x \in \mathcal{M}^+$ and 
$n \geq 1$, such that $\sigma(d_xT^n) > 3$, there exists an open ergodic neighbourhood 
$\mathcal{U}(x)$, that lies $(\operatorname{mod} 0)$ in one ergodic component. 
\end{local}

Chernov postulated in \cite{Ch93} an equivalent condition to Condition \ref{condi5}. Denote by 
$W^{u}(x)$ resp. $W^{s}(x)$ the unstable resp. stable manifolds at point $x$.
\begin{con}[Transversality]\label{condi6}
For $\mu_{|_{\mathcal{S}^{\pm}}}$ -a.e. $x$, the stable subspace $W^{s}(x)$ resp. unstable subspace 
$W^{u}(x)$ is transversal to $\mathcal{S}^-$ resp. $\mathcal{S}^+$. 
\end{con}
\begin{lemma}\label{equivalence}
The proper alignment condition and the transversality condition are pointwise equivalent.
\end{lemma}
\begin{proof}
Assume that at $x \in \mathcal{S}^-$ the singularity manifold is not properly aligned. Then 
$d_xT^n v_x^- \notin \mathcal{C}(T^nx)$, for all $n \geq 0$. This implies that 
$v_x^- \in \mathcal{T}W^s(x)$. Since $\mathcal{T}W^s(x)$ is a Lagrangian subspace (see e.g. \cite[Lemma 4]{W88})
$\omega(v_x^-,v) = 0$, for all $v \in \mathcal{T}W^s(x)$. Additionally, 
$\omega(v_x^-,w) = 0$, for all $w \in \mathcal{T}_x\mathcal{S}^-$, because the characteristic line 
is the $\omega$-orthogonal complement of $\mathcal{T}_x\mathcal{S}^-$. But this yields 
$\mathcal{T}W^s(x) \subset \mathcal{T}_x\mathcal{S}^-$, hence, the singularity manifold 
$\mathcal{S}^-$ is not transversal to the stable manifold $W^s(x)$ at $x$.

Assume that at point $x \in \mathcal{S}^-$, the singularity manifold and the stable manifold 
$W^s(x)$ are not transversal but still properly aligned, i.e. 
$\mathcal{T}W^s(x) \subset \mathcal{T}\mathcal{S}^-$ and $v_x^- \cap \mathcal{T}W^s(x) = \varnothing$. 
Since transversality is not satisfied and $v_x^-$ is the characteristic line, we have 
$\omega(v_x^-,v) = 0$, for all $v \in \mathcal{T}W^s(x)$. This means, that 
$v_x^- \in (\mathcal{T}W^s(x))_{\omega}^{\perp}$, where $(\mathcal{T}W^s(x))_{\omega}^{\perp}$ 
is the $\omega$-orthogonal complement of $\mathcal{T}W^s(x)$. But $\mathcal{T}W^s(x)$ is 
a Lagrangian subspace and, thus, $(\mathcal{T}W^s(x))_{\omega}^{\perp} = \mathcal{T}W^s(x)$. 
Hence, $v_x^- \in \mathcal{T}W^s(x)$, which results in a contradiction.
\end{proof}
Even though the proper alignment condition and the transversality condition are pointwise 
equivalent it is presently unclear whether it is enough for the local ergodic theorem (in the 
Liverani-Wojtkowski framework) to 
hold by considering the validity of the proper alignment condition only on a set of full measure 
with respect to the measure $\mu_{|_{\mathcal{S}^{\pm}}}$.

\subsubsection{The current state of proper alignment}\label{state}
There has been a substantial misconception whether the system of falling balls is properly aligned 
or not. In brief, the correct answer to this question is that on some part of the singularity 
manifold the system is properly aligned and on the complementary part we simply do not know. 
The latter affects only the singularity manifolds $\mathcal{S}_{1,2}^{\pm}$, since every point 
on $\mathcal{S}_{3,1}^{\pm}$ is properly aligned.
The original formulation of the proper alignment condition in \cite{LW92} is more restrictive
than the one stated above. Namely, it demands the characteristic line $v_x^-$ resp. $v_x^+$ to lie in 
$\mathcal{C}(x)$ resp. $\mathcal{C}^{\prime}(x)$ for every point of the singularity 
manifolds. 
Below of the original proper alignment condition it says (\cite[p. 37]{LW92})
\begin{quote}
It will be clear from the way in which the proper alignment of singularity sets is 
used in Section 12 that it is sufficient to assume that there is $N$ such that 
$T^N\mathcal{S}^-$ and $T^{-N}\mathcal{S}^+$ are properly aligned. 
\end{quote}
In Section 12 of \cite{LW92} the authors remind the reader, that, in their constructive argument, 
the size of the neighbourhood $\mathcal{U}(x)$, appearing in the Local Ergodic Theorem, was chosen 
small enough, such that $\mathcal{U}(x) \cap \mathcal{S}_N^- = \varnothing$. 
Due to the regularity of singularity manifolds (see Condition \ref{condi3}), for every $M > N$, there 
exists a finite $p = p(M) > 0$, such that $\bigcup_{i = N}^M T^i\mathcal{S}^- = \bigcup_{k = 1}^p I_k$, 
where $I_k$ are compact submanifolds (see Definition \ref{reg}). 
In the proof of Proposition 12.2, Liverani and Wojtkowski make use of the fact, that \textbf{every} point 
$x \in I_k$ is properly aligned (see \cite[p. 57]{LW92}). Hence, the relaxed version of the proper 
alignment condition (see Condition \ref{condi5}) is justified.

The authors continue (see \cite[p. 37]{LW92}) with the following assertion
\begin{quote}
We will show, in section 14, that for the system of falling balls even this weaker property 
[see Condition \ref{condi5}] fails.
\end{quote}
The content of the last quotation is wrong. We will now illustrate what Liverani and Wojtkowski really 
did in section 14: The argument is carried out for the singularity manifold $\mathcal{S}_{1,2}^-$. 
The characteristic line at point $x \in \mathcal{S}_{1,2}^-$ is given by
\begin{align}
v_x^- = \{(\delta q,\delta p) \in \mathcal{T}_x\mathcal{S}_{1,2}^-:\ 
&\delta q_1 = \delta q_2 = \delta q_3 = 0,\nonumber\\ 
&\sum_{i=1}^3 \delta p_i = 0,\ \sum_{i=1}^3 \frac{p_i\delta p_i}{m_i} = 0,\
\frac{p_1}{m_1} \leq \frac{p_2}{m_2} \leq \frac{p_3}{m_3}\}.\nonumber
\end{align}
The restrictions of the momenta follow from 
$\mathcal{S}_{1,2}^- \subset \mathcal{M}_2^+ \cap \mathcal{M}_3^+$. We will look at the set 
of momenta in a little bit more detail:  Without loss of generality let $t_0  < t_1$, 
$x = x(t_0) \in \mathcal{S}_{1,2}^-$ and $Tx = x(t_1) \in \mathcal{M}_1^+$. Since 
$p_1^+(t_0) / m_1 \leq p_2^+(t_0) / m_2 \leq p_3^+(t_0) / m_3$, applying the equations of motion
(\ref{equation2}) yields $p_1^-(t_1) / m_1 \leq p_2^-(t_1) / m_2 \leq p_3^-(t_1) / m_3$. 
Due to $x(t_1) \in \mathcal{M}_1^+$, we have $p_1^-(t_1) / m_1 < 0$. Incorporating the latter, 
we $(\operatorname{mod} 0)$ partition the set of eligible momenta at time $t_1$ into the subsets 
\begin{align}
&\mathsf{Mom}_1(q(t_1),p^-(t_1)) = \Bigl\{\frac{p_1^-(t_1)}{m_1} < 0 \leq \frac{p_2^-(t_1)}{m_2} \leq 
\frac{p_3^-(t_1)}{m_3}\Bigr\},\nonumber\\
&\mathsf{Mom}_2(q(t_1),p^-(t_1)) = \Bigl\{\frac{p_1^-(t_1)}{m_1} < \frac{p_2^-(t_1)}{m_2} \leq 0 \leq 
\frac{p_3^-(t_1)}{m_3}\Bigr\},\nonumber\\
&\mathsf{Mom}_3(q(t_1),p^-(t_1)) = \Bigl\{\frac{p_1^-(t_1)}{m_1} < \frac{p_2^-(t_1)}{m_2} \leq \frac{p_3^-(t_1)}{m_3} \leq 0\Bigr\}.\nonumber
\end{align}
Using again the equations of motion, we obtain in time $t_0$
\begin{align}
&\mathsf{Mom}_1(q(t_0),p^+(t_0)) = \Bigl\{\frac{p_1^+(t_0)}{m_1} < t_1-t_0 \leq \frac{p_2^+(t_0)}{m_2} \leq \frac{p_3^+(t_0)}{m_3}\Bigr\},\nonumber\\
&\mathsf{Mom}_2(q(t_0),p^+(t_0)) = \Bigl\{\frac{p_1^+(t_0)}{m_1} < \frac{p_2^+(t_0)}{m_2} \leq t_1-t_0 \leq \frac{p_3^+(t_0)}{m_3}\Bigr\},\nonumber\\
&\mathsf{Mom}_3(q(t_0),p^+(t_0)) = \Bigl\{\frac{p_1^+(t_0)}{m_1} < \frac{p_2^+(t_0)}{m_2} \leq \frac{p_3^+(t_0)}{m_3} \leq t_1-t_0\Bigr\}.\nonumber
\end{align}

Observe that all the momenta can only be simultaneously negative on the set $\mathsf{Mom}_3(q(t_0),p^+(t_0))$.

The quadratic form $Q$ of the contracting cone field in coordinates $(q,p)$ equals
\begin{align}
Q(\delta q, \delta p) = \sum_{i=1}^3 \delta q_i \delta p_i + \frac{p_i(\delta p_i)^2}{m_i^2}.\nonumber
\end{align}
Inserting $v_x^-$ into $Q$ results in
\begin{align}\label{qchar}
Q(v_x^-) = \sum_{i=1}^3 \frac{p_i(\delta p_i)^2}{m_i^2}.
\end{align}
Ths singularity manifold $\mathcal{S}_{1,2}^-$ at point $x$ is properly aligned if and only if 
$Q(v_x^-) \geq 0$. It is easy to see, that each of the sets $\mathsf{Mom}_i(q(t_0),p^+(t_0))$ contains 
a subset on which $\mathcal{S}_{1,2}^-$ is not properly aligned, i.e. $Q(v_x^-) < 0$. 
Hence, depending on the point $x \in \mathcal{S}_{1,2}^-$, (\ref{qchar}) can obtain non-negative and 
negative values on every set $\mathsf{Mom}_i(q(t_0),p^+(t_0))$.

Additionally note, that the image of the characteristic line is the characteristic line of the image, 
i.e. 
\begin{align}\label{image}
d_xT^n v_x^- = v_{T^nx}^-.
\end{align}
Combining this with the fact, that $d_xT$ is Q-monotone for 
every point $x \in \mathcal{M}^+$ (see Definition \ref{Defq}.3) we obtain, that once a point is 
properly aligned, it remains properly aligned.

We summarize, that on some parts of $\mathcal{S}_{1,2}^-$ the system of falling balls 
is properly aligned and on the complement we do not know, since an iterate of the characteristic 
line could very well be mapped into the contracting cone field. This is exactly what 
Liverani and Wojtkowski prove in section 14. More importantly, they do \textbf{not} examine 
whether any iterate of $v_x^-$ gets mapped into the contracting cone field or not. This is currently 
not known. 

\subsubsection{Iterates of the characteristic line}\label{iterates}
The Main Theorem allows us to compare the measure of iterated singular points, which are not properly 
aligned, to not properly aligned points of the iterated singularity manifold. For this, an immediate 
consequence of the Main Theorem is, that the monodromy matrix $d_xT$ is eventually strictly Q-monotone 
for every point (see e.g. (\ref{point3}) in Theorem \ref{thm1}), i.e. for every $x \in \mathcal{M}^+$, 
there exists $k = k(x) \geq 1$: $Q(d_xT^kv) > Q(v)$, for all $v \in L_1 \oplus L_2$. Define, 
for $n \geq 1$, the sets 
\begin{align}
A(n,\mathcal{S}_{1,2}^-) &= \{x \in \mathcal{S}_{1,2}^-:\ Q(v_x^-)<0,\ Q(d_xT^nv) > Q(v), 
\forall\ v \in L_1 \oplus L_2\},\nonumber\\
\bigcup_{n \geq 1} A(n,\mathcal{S}_{1,2}^-) &= A(\mathcal{S}_{1,2}^-).\nonumber
\end{align}
The sets $A(n,\mathcal{S}_{1,2}^-)$ consist of all points in $\mathcal{S}_{1,2}^{-}$, 
which are not properly aligned and have an 
eventually strictly Q-monotone monodromy matrix after $n$ steps. We remark, that the sets 
$A(n,\mathcal{S}_{1,2}^-)$ are empty for small values of $n$. Once 
$A(n,\mathcal{S}_{1,2}^-) \neq \varnothing$, the Q-monotonicity of $d_xT$ for every point 
implies that $A(n,\mathcal{S}_{1,2}^-) \subseteq A(n+1,\mathcal{S}_{1,2}^-)$. 
We split $A(n,\mathcal{S}_{1,2}^-)$ up into $A^+(n,\mathcal{S}_{1,2}^-) \cup 
A^-(n,\mathcal{S}_{1,2}^-)$, where
\begin{align}
A^+(n,\mathcal{S}_{1,2}^-) = 
\{x \in A(n,\mathcal{S}_{1,2}^-):\ Q(d_xT^nv_x^-) = Q(v_{T^nx}^-) \geq 0\},\nonumber\\
A^-(n,\mathcal{S}_{1,2}^-) = 
\{x \in A(n,\mathcal{S}_{1,2}^-):\ Q(d_xT^nv_x^-) = Q(v_{T^nx}^-) < 0\}.\nonumber
\end{align}
$A^+(n,\mathcal{S}_{1,2}^-)$ are the points which become properly aligned after at most 
$n$ iterates and $A^-(n,\mathcal{S}_{1,2}^-)$ are the points which remain not properly aligned 
after $n$ iterates. Keep in mind, that $A^+(n,\mathcal{S}_{1,2}^-)$ can be empty for some 
$n$. Due to the eventually strict Q-monotonicity of $d_xT$, we have
\begin{align}
Q(d_{T^nx}T^{-n} v_{T^nx}^-) < Q(v_{T^nx}^-),\ \forall\ T^nx \in T^n A(n,\mathcal{S}_{1,2}^-).\nonumber
\end{align}
Using the last statement together with (\ref{image}), we obtain
\begin{align}
T^n A^-(n,\mathcal{S}_{1,2}^-) = T^n A(n,\mathcal{S}_{1,2}^-) \cap 
A(T^n\mathcal{S}_{1,2}^-) \subset A(T^n\mathcal{S}_{1,2}^-).\nonumber
\end{align}
The latter conclusion holds for every $n$ 
with $A(n,\mathcal{S}_{1,2}^-) \neq \varnothing$ 
and can be repeated for future iterates. This proves, that the set $T^n A(n,\mathcal{S}_{1,2}^-)$ 
of not properly aligned points after $n$ iterations is strictly contained in the set $A(T^n\mathcal{S}_
{1,2}^-)$ of not properly aligned points of the singularity manifold $T^n\mathcal{S}_{1,2}^-$.

However, the size of $T^n A(n,\mathcal{S}_{1,2}^-)$ and whether there exists a fixed $N \geq 1$, for 
every $n \geq 1$, such that $T^N A(n,\mathcal{S}_{1,2}^-) = \varnothing$, remains unknown.

\subsection{Abudance of least expanding points}
Liverani and Wojtkowski require the point in the local ergodic theorem to have least expansion 
coefficient larger than three. However, after their formulation of the local ergodic theorem 
they point out (see \cite[p. 39]{LW92}) that there is no loss in generality in actually demanding 
that the least expansion coefficient is only larger than one. The reason for this is due to the 
fact, that the set of points with non-zero Lyapunov exponents has full measure 
(see \cite{S96}, \cite{W98}). We quote
\begin{quote}
Let us note that the conditions of the last theorem are satisfied for almost all points 
$p \in \mathcal{M}$. Indeed, let
\begin{align}
M_{n,\epsilon} = \{p \in \mathcal{M}|\sigma(D_pT^n) > \epsilon\}.\nonumber
\end{align}
Since almost all points are strictly monotone, then
\begin{align}
\bigcup_{n=1}^{+\infty}\bigcup_{\epsilon > 0}\mathcal{M}_{n,\varepsilon}\nonumber
\end{align}
has full measure. By the Poincar\'e Recurrence Theorem and the supermultiplicativity of the 
coefficient $\sigma$ we conclude that
\begin{align}
\bigcup_{n=1}^{+\infty}\mathcal{M}_{n,3}\nonumber
\end{align}
has also full measure.
\end{quote}

\begin{definition}
	A point $x \in \mathcal{M}^+$ is called least expanding, if there exists an $n \geq 1$, such that 
	$\sigma(d_xT^n) > 1$. 
\end{definition}
Once local ergodicity is established we know that every ergodic component is (mod 0) open. 
To obtain a single ergodic component one needs to verify
\begin{theorem}[Abudance of least expanding points]\label{abundance}
	The set of least expanding points has full measure and is arcwise connected.
\end{theorem}
More precisely, this implies, that one can connect any two least expanding points by a curve, 
which lies completely in the set of least expanding points. 
Consequently the points on the curve can be chosen in such a way, that the open 
neighbourhoods, from the local ergodic theorem, intersect pairwise on a set of positive measure. 
Hence, there can only be one ergodic component. For a more detailed proof see e.g.
\cite[p. 151 - 152]{CM06}.

\section{Strict unboundedness - Part I}
In this section we will begin with the proof of the strict unboundedness of the sequence $(d_{T^n x}T)_{n \in \mathbb{N}}$, 
for every $x \in \mathcal{M}^+$. Due to \cite[Theorem 6.8]{LW92} we have the following equivalence

\begin{theorem}\label{thm1}
	For every $x \in \mathcal{M}^+$, the sequence $(d_{T^n x}T)_{n \in \mathbb{N}}$ is strictly unbounded if and only if
	\begin{subequations}
		\begin{eqnarray}
		&&\text{For every}\ x \in \mathcal{M}^+,\ \text{the sequence}\ (d_{T^n x}T)_{n \in \mathbb{N}}\ \text{is unbounded.}\label{point2}\\
		&&\text{For every}\ x \in \mathcal{M}^+,\ \text{there exist}\ k_1, k_2 \in \mathbb{N},\ 
		\text{such that } Q(d_xT^{k_1}(\delta \xi, 0)) > 0\label{point3}\\ 
		&&\text{and}\ Q(d_xT^{k_2}(0, \delta \eta)) > 0,\ \text{for all}\ (\delta \xi, 0) \in L_1, (0, \delta \eta) \in L_2.\nonumber
		\end{eqnarray}
	\end{subequations}
\end{theorem}

We will prove the strict unboundedness by equivalently proving properties (\ref{point2}) and (\ref{point3}).

The most important ingredient for (\ref{point2}) is the following
\begin{theorem}\label{thm2}
	There exists a positive constant $\Lambda > 0$, such that for all $x \in \mathcal{M}^+$, 
	there exists a sequence of strictly increasing positive  
	integers $(n_k)_{k \in \mathbb{N}} = (n_k(x))_{k \in \mathbb{N}}$ and for all 
	$(0, \delta \eta) \in L_2:$
	\begin{align}\label{esta1}
	Q(d_{T^{n_{2k-2}}x}T^{n_{2k-1} - n_{2k-2}}(0, \delta \eta)) > \Lambda \|(0, \delta \eta)\|^2.
	\end{align}
\end{theorem}

In fact, we will prove, that $dT^{n_{2k-1} - n_{2k-2}}$ either equals $d\Phi_{1,2}d\Phi_{2,3}$ or 
$d\Phi_{2,3}d\Phi_{1,2}$.
Recursively define $(\delta \xi_n, \delta \eta_n) = dT (\delta \xi_{n-1}, \delta \eta_{n-1})$, with 
$(\delta \xi_0, \delta \eta_0) = (\delta \xi, \delta \eta)$ and $q_n = Q(\delta \xi_n, \delta \eta_n)$. 
From \cite{W90a} we know, that $d_xT$ is Q-monotone for every $x \in \mathcal{M}^+$, therefore, $q_{n+1} \geq q_n$. 
Hence, in order to prove $\lim_{n \to +\infty} q_n = +\infty$, it is enough to prove this 
divergence along a subsequence $(q_{n_{2k-1}})_{k \in \mathbb{N}}$. We define this subsequence by setting
\begin{align}\label{thisisit}
q_{n_{2k-1}} = Q(d_{T^{n_{2k-2}}x}T^{n_{2k-1} - n_{2k-2}}(\delta \xi_{n_{2k-2}}, \delta \eta_{n_{2k-2}})).
\end{align}

We will postpone the proof of Theorem \ref{thm2} and property (\ref{point3}) to section 8, as they will both 
follow from our analysis of a particle moving inside a wedge (see section 7). Here we will show how Theorem 
\ref{thm2} is utilized to prove the unboundedness property (\ref{point2}). In fact, (\ref{point2}) will 
be obtained by using the estimate from Theorem \ref{thm2} in a modified version of the unboundedness proof 
in \cite[p. 32 - 33]{LW92}. Beforehand we need to take some preparatory steps. 
\begin{prop}\label{prop5}
	For every $x \in \mathcal{M}^+$, we have
	\begin{align}\label{41}
	q_{n_{2k+1}} > q_{n_{2k}} + \Lambda\|(0, \delta \eta_{n_{2k}})\|^2.
	\end{align}
\end{prop}

\begin{proof}
	Without loss of generality let 
	$d_{T^{n_{2k}}x}T^{n_{2k+1} - n_{2k}}$ be the product of $d\Phi_{1,2}d\Phi_{2,3}$. Using (\ref{esta1}), we estimate
	\begin{eqnarray}
	q_{n_{2k+1}} &=& Q(d_{T^{n_{2k}}x}T^{n_{2k+1} - n_{2k}}(\delta \xi_{n_{2k}}, \delta \eta_{n_{2k}}))\nonumber\\[0.2cm]
	&=& Q\Bigl(
	\begin{pmatrix}
	M_1M_2 & M_1U_2 + U_1M_2^T\\
	0 & M_1^T M_2^T
	\end{pmatrix}
	\begin{pmatrix}
	\delta \xi_{n_{2k}}\\
	\delta \eta_{n_{2k}}
	\end{pmatrix}\nonumber
	\Bigr)\nonumber \\[0.2cm]
	&=& \langle M_1M_2 \delta \xi_{n_{2k}} + (M_1U_2 + U_1M_2^T)\delta \eta_{n_{2k}},\ M_1^T M_2^T \delta \eta_{n_{2k}} \rangle 
	\nonumber\\[0.2cm]
	&=& \langle M_1M_2 \delta \xi_{n_{2k}},\ M_1^T M_2^T \delta \eta_{n_{2k}} \rangle + 
	Q\Bigl(
	\begin{pmatrix}
	M_1M_2 & M_1U_2 + U_1M_2^T\\
	0 & M_1^T M_2^T
	\end{pmatrix}
	\begin{pmatrix}
	0\\
	\delta \eta_{n_{2k}}
	\end{pmatrix} \Bigr)\nonumber \\[0.2cm]
	&>& \langle \delta \xi_{n_{2k}},\ \delta \eta_{n_{2k}}\rangle + \Lambda \|(0, \delta \eta_{n_{2k}})\|^2\nonumber\\[0.2cm] 
	&=& q_{n_{2k}} + \Lambda \|(0, \delta \eta_{n_{2k}})\|^2.\nonumber
	\end{eqnarray}
\end{proof}

\begin{prop}\label{prop6}
	Let $(a_{n_k})_{k \in \mathbb{N}}$ be a sequence of positive numbers and $C$ a positive constant. If
	\begin{align} 
	\sum_{i = 0}^{+\infty}a_{n_{2i}}= +\infty \text{ then } \sum_{k = 0}^{+\infty}
	\frac{a_{n_{2k}}}{C +\sum_{i=0}^{k} a_{n_{2i}}} = +\infty.\nonumber
	\end{align} 
\end{prop}

\begin{proof}
	For $1 \leq j \leq l$, we have
	\begin{align}
	\sum_{k = j}^{l}\frac{a_{n_{2k}}}{C + \sum_{i=0}^{k} a_{n_{2i}}} > 
	\frac{\sum_{k = j}^{l}a_{n_{2k}}}{C + \sum_{i=0}^{j-1} a_{n_{2i}} + \sum_{i=j}^{l} a_{n_{2i}}}
	\to 1,\ \text{as}\ l \to +\infty.\nonumber
	\end{align}
	The tail of the series does not tend to zero, hence the series diverges.
\end{proof}

Consider the subsequence $(q_{n_{2k-1}})_{k \in \mathbb{N}}$ introduced in (\ref{thisisit}). Since $\prod_{k = 1}^{+ \infty} q_{n_{2k-1}} / q_{n_{2k-2}} = + \infty$ implies $\lim_{n \to +\infty}q_{n_{2k-1}}  = + \infty$, we will estimate
\begin{align}
\prod_{k = 1}^{+ \infty} \frac{q_{n_{2k-1}}}{q_{n_{2k-2}}} \geq \prod_{k = 1}^{+ \infty} 1 + r_k, \nonumber
\end{align}
and further prove, that $\sum_{k = 1}^{+ \infty} r_k = + \infty$, which yields the unboundedness.

Before we start with the proof of property (\ref{point2}) we need to recall and calculate some 
preliminary necessities: 
\begin{enumerate}
\item From the definition of the monodromy maps, 
we immediately obtain
\begin{align}\label{maps}
\begin{array}{rcl}
d\Phi_{0,1}(\delta \xi_{n-1}, \delta \eta_{n-1}) &= &
\begin{pmatrix}
\delta \xi_{n-1} \\
B \delta \xi_{n-1} + \delta \eta_{n-1}
\end{pmatrix}
=
\begin{pmatrix}
\delta \xi_n\\
\delta \eta_n
\end{pmatrix},\\[0.4cm]
d\Phi_{i,i+1}(\delta \xi_{n-1}, \delta \eta_{n-1}) &= & 
\begin{pmatrix}
M_i \delta \xi_{n-1} + U_i \delta \eta_{n-1} \\
M_i^T \delta \eta_{n-1}
\end{pmatrix}
=
\begin{pmatrix}
\delta \xi_n\\
\delta \eta_n
\end{pmatrix},\ i= 1,2.
\end{array}
\end{align}

\item Cheng and Wojtkowski introduced 
in \cite{ChW91} the norm
\begin{align}
\|\delta \xi\|_{CW}^2 = \sum_{i = 1}^2 \frac{(\delta \xi_{i+1} - \delta \xi_{i})^2}{m_i}.\nonumber
\end{align}
The maps $M_i$ are invariant with respect to this norm, i.e. 
\begin{align}\label{CW}
\|M_i \delta \xi\|_{CW} = \|\delta \xi\|_{CW}. 
\end{align}

\item The equivalence of norms gives us constants $D_1, D_2> 0$, such that
\begin{align}\label{p1}
&D_1 \|\delta \xi\|_{\max} \leq \|\delta \xi\|_{CW} \leq D_2 \|\delta \xi\|_{\max},
%&D_3 \|\delta \eta\|_{\max} \leq \|\delta \eta\|_2 \leq D_4 \|\delta \eta\|_{\max},
\end{align}
where $\|\cdot\|_{\max}$ denotes the maximum norm. 

\item Using the definitions of the Hamiltonian and the terms $\alpha_i$ (\ref{alpha}), we calculate
\begin{align}
	\max\{\alpha_1,\alpha_2\} \leq \frac{4\sqrt{2c}m_1^3}{m_3^2\sqrt{m_3}},
\end{align}
where $c > 0$ is the energy of the system.
\item Let $(i,i+1)$, $i = 0,1,2$, stand for a collision of ball $i$ with ball $i+1$, 
i.e. when $q_i = q_{i+1}$. When $i = 0$ the system experiences a collision with the floor.
\end{enumerate}
\begin{proof}[Proof of property (\ref{point2})]
The proof is based on the scheme given in \cite[p. 32 - 33]{LW92}.\\
We first give an estimate for $\|\delta \xi_{n_{2k-1}}\|_{CW}$ in 
between points $T^{n_{2k-1}}x$ and $T^{n_{2k-2}}x$. Without loss of generality we set $d_{T^{n_{2k-2}}x}T^{n_{2k-1} - n_{2k-2}}$ 
to be the product of $d\Phi_{1,2}d\Phi_{2,3}$ for every $k \in \mathbb{N}$. We estimate
\begin{eqnarray}\label{p2}
\|\delta \xi_{n_{2k-1}}\|_{CW} &=& \|M_1M_2 \delta \xi_{n_{2k-2}} + (M_1U_2 + U_1M_2^T)\delta \eta_{n_{2k-2}}\|_{CW}\nonumber\\
&\leq& \|\delta \xi_{n_{2k-2}}\|_{CW} + \|(M_1U_2 + U_1M_2^T) \delta \eta_{n_{2k-2}}\|_{CW}\nonumber\\
&\leq& \|\delta \xi_{n_{2k-2}}\|_{CW} + D_2 \|
\begin{pmatrix}
\alpha_1 & (1+\gamma_1)\alpha_2 + (1-\gamma_2)\alpha_1 \\
0 & \alpha_2
\end{pmatrix}
\delta \eta_{n_{2k-2}}\|_{\max}\nonumber\\
&\leq& \|\delta \xi_{n_{2k-2}}\|_{CW} + D_2 3 \max\{\alpha_1, \alpha_2\} \|\delta\eta_{n_{2k-2}}\|_{\max}\nonumber\\
&\leq& \|\delta \xi_{n_{2k-2}}\|_{CW} + \frac{D_212\sqrt{2c}m_1^3}{m_3^2\sqrt{m_3}} \|\delta\eta_{n_{2k-2}}\|_{\max}
\end{eqnarray}
We abbreviate the constant factor in the last inequality by
\begin{align}
K =  \frac{D_212\sqrt{2c}m_1^3}{m_3^2\sqrt{m_3}}.\nonumber
\end{align}
In between points $T^{n_{2k}}x$ and $T^{n_{2k-1}}x$ we have one of the following situations: Either a floor collision occurs, 
in which $\|\delta \xi_{n_{2k}}\|_{CW} = \|\delta \xi_{n_{2k}-1}\|_{CW}$ or a ball to ball collision occurs, in which 
$\|\delta \xi_{n_{2k}}\|_{CW} \leq \|\delta \xi_{n_{2k}-1}\|_{CW} + \|U_{\kappa(n_{2k}-1)}\delta \eta_{n_{2k}-1}\|_{CW}$ (see (\ref{maps})). Thereby, $\kappa: \mathbb{N} \to \{1,2\}$, depends on the point and describes whether we have a (1,2) or (2,3) collision.
Combining this with (\ref{p2}) we obtain
\begin{align}\label{use}
\|\delta \xi_{n_{2k}}\|_{CW} \leq \|\delta \xi_{n_0}\|_{CW} + \sum_{i=1}^{k}\sum_{j \in I_i} \|U_{\kappa(n_{2i}-j)}\delta \eta_{n_{2i}-j}\|_{CW} + K\sum_{i=1}^{k}\|\delta \eta_{n_{2i-2}}\|_{CW},
\end{align}
where $\left\vert{I_i}\right\vert$ are the number of ball to ball collisions happening between points 
$T^{n_{2i}}x$ and $T^{n_{2i-1}}x$. If $\left\vert{I_i}\right\vert = 0$, then we set 
$\|U_{\kappa(n_{2i})}\delta \eta_{n_{2i}}\|_{CW} = 0$.

The Cauchy-Schwarz inequality gives us
\begin{align}
q_{n_k} = \langle \delta \xi_{n_k}, \delta \eta_{n_k}\rangle \leq \|\delta \xi_{n_k}\|\|\delta \eta_{n_k}\|,\nonumber
\end{align}
which yields
\begin{align}\label{p4}
\|\delta \eta_{n_k}\| \geq \frac{q_{n_k}}{\|\delta \xi_{n_k}\|}.
\end{align}
From Proposition \ref{prop5} and the Cauchy-Schwarz inequality, we get
\begin{align}
q_{n_{2k+1}} &> q_{n_{2k}} + \Lambda \|\delta \eta_{n_{2k}}\|_{\max}^2\nonumber\\
&\geq q_{n_{2k}} +  \Lambda \|\delta \eta_{n_{2k}}\|_{\max} \frac{q_{n_{2k}}}{\|\delta \xi_{n_{2k}}\|_{\max}}\nonumber\\
&\geq q_{n_{2k}} \Bigl( 1 + \Lambda \frac{D_1\|\delta \eta_{n_{2k}}\|_{\max}}{\|\delta \xi_{n_{2k}}\|_{CW}}\Bigr).\nonumber
\end{align}
Utilizing the above, we estimate
\begin{align}
\frac{q_{n_{2k+1}}}{q_{n_{2k}}} \geq 1 + \frac{\Lambda D_1\|\delta \eta_{n_{2k}}\|_{\max}}{\|\delta \xi_{n_0}\|_{CW} + \sum_{i=1}^{k}\sum_{j \in I_i} \|U_{\kappa(n_{2i}-j)}\delta \eta_{n_{2i}-j}\|_{CW} + K\sum_{i=1}^{k}\|\delta \eta_{n_{2i-2}}\|_{CW}}. \nonumber
\end{align}

Let 
\begin{align}
r_k = \frac{\Lambda D_1\|\delta \eta_{n_{2k}}\|_{\max}}{\|\delta \xi_{n_0}\|_{CW} + \sum_{i=1}^{k}\sum_{j \in I_i} \|U_{\kappa(n_{2i}-j)}\delta \eta_{n_{2i}-j}\|_{CW} + K\sum_{i=1}^{k}\|\delta \eta_{n_{2i-2}}\|_{CW}}.
\nonumber
\end{align}
Without loss of generality assume\footnote{If the sum is infinite, then we can apply the argument in  
\cite[p. 32 - 33]{LW92} directly. The key point is, that we do not have control over this sum, so we 
assume the worst case, namely, its finiteness.} that the sum $\sum_{i=1}^{+\infty}\sum_{j \in I_i} \|U_{
\kappa(n_{2i}-j)}\delta \eta_{n_{2i}-j}\|_{CW}$ is finite. The only thing left to show is that $\sum_{k 
= 1}^{+ \infty}r_k = +\infty$.  In view of Proposition \ref{prop6}, it will follow once we prove, that 
$\sum_{i = 0}^{+\infty}\|\delta \eta_{n_{2i}}\|_{\max} = +\infty$. Assume on the contrary, that this is 
not true. 
Then, by (\ref{p2}), the sequence $(\|\delta \xi_{n_{2k-1}}\|_{CW})_{k \in \mathbb{N}}$ is bounded from 
above. This 
and (\ref{p4}) imply, that $(\|\delta \eta_{n_{2k}}\|_{\max})_{k \in \mathbb{N}}$ is bounded away from zero, which contradicts our assumption. This yields the unboundedness.
\end{proof}

\section{Particle falling in a wedge}
Wojtkowski analyzed in \cite{W98} the hyperbolicity of a particle moving along parabolic trajectories 
in a variety of wedges. The particle is subject to constant acceleration and collides with the walls 
of the wedge. We adopt his notation and call such a system particle falling in a wedge- or, 
abbreviated, PW system.
Heuristically speaking, for special wedges, namely simple ones, the PW system is equivalent to a 
falling balls system (or FB system) with particular masses. 
After introducing the basic setup in three dimensions we are going to recall and expand some of 
the results in \cite{W98} in order to prove Theorem \ref{thm2} and property (\ref{point3}) 
in section \ref{eight}.

Let $E$ be the three dimensional Euclidean space. For three linearly independent vectors 
$\{e_1,e_2,e_3\}$ we define the wedge $W(e_1,e_2,e_3) \subset E$ by
\begin{align}
W(e_1,e_2,e_3) = \{e \in E:\ e = \lambda_1e_1 + \lambda_2e_2 + \lambda_3e_3,\ \lambda_i \geq 0,\
i = 1,2,3\}.\nonumber
\end{align}
The set of vectors $\{e_1,e_2,e_3\}$ are called the generators of the wedge. We denote by 
$S(e_1,\ldots,e_i)$, $1 \leq i \leq 3$, the linear subspace spanned by the linearly independent vectors 
$\{e_1,\ldots,e_i\}$. A three dimensional wedge is called simple, if the generators can be ordered in 
such a way that the orthogonal projection of $e_1$ resp. $e_2$ onto $S(e_2,e_3)$ resp. $S(e_3)$ is a 
positive multiple of $e_2$ resp. $e_3$. The simplicity of a wedge can be verified with the following
\begin{prop}[Proposition 2.3, \cite{W98}]\label{prop2.3}
Let $\{e_1,e_2,e_3\}$ be a set of linearly independent unit vectors. The wedge $W(e_1,e_2,e_3)$ 
is simple if and only if 
\begin{subequations}
		\begin{eqnarray}
		&&\langle e_i, e_{i+1} \rangle > 0,\ i =1,2,\label{point4}\\
		&&\langle e_1, e_3 \rangle = \langle e_1, e_2 \rangle \langle e_2, e_3 \rangle.\label{point5}
		\end{eqnarray}
	\end{subequations} 
\end{prop}
The angles $\alpha_i = \sphericalangle (e_i,e_{i+1})$, $i = 1,2$,
completely determine the geometry of the wedge. In a simple wedge the angles satisfy 
$0 < \alpha_i < \frac{\pi}{2}$ and if $\{e_1,e_2,e_3\}$ are unit vectors we have
\begin{align}\label{2.1}
\cos \alpha_i = \langle e_i, e_{i+1} \rangle.
\end{align} 
We also give another geometric characterization of the wedge by introducing a second pair of angles $
\beta_1$, $\beta_2$. 
Thereby, $\beta_i$ is the angle between subspaces $S(e_i, e_{i+2})$ and $S(e_{i+1}, e_{i+2})$, 
where for $i = 2$, we set $\beta_2 = \alpha_2$. If the wedge is simple, they satisfy 
$0 < \beta_i < \frac{\pi}{2}$. 
The relation between $\beta_1$ and $\alpha_1, \alpha_2$ is given by
\begin{align}\label{2.2}
\tan \beta_1 = \frac{\tan \alpha_1}{\sin \alpha_2}.
\end{align}

Consider the FB system from Section 2. Its Hamiltonian is given by 
$H(q,p) = \frac{1}{2}\langle Kp,p \rangle + \langle c_1, q\rangle$, 
$K = diag(\frac{1}{m_1},\frac{1}{m_2},\frac{1}{m_3})$, 
$c_1 = (m_1, m_2, m_3)$. Thereby, $K$ is the diagonal matrix with 
diagonal entries $\frac{1}{m_1},\frac{1}{m_2},\frac{1}{m_3}$. The unit vectors
\begin{align}
e_1 = \frac{1}{\sqrt{3}}
\begin{pmatrix}
1\\
1\\
1
\end{pmatrix},
e_2 = \frac{1}{\sqrt{2}}
\begin{pmatrix}
0\\
1\\
1
\end{pmatrix},
e_3 = 
\begin{pmatrix}
0\\
0\\
1
\end{pmatrix}\nonumber
\end{align}
span the configuration space
\begin{align}
W_q(e_1,e_2,e_3) = \{(q_1,q_2,q_3) \in \mathbb{R}^3: 0 \leq q_1 \leq q_2 \leq q_3\}.\nonumber
\end{align}
It carries the natural Riemannian metric given by the kinetic energy 
$\langle K \cdot, \cdot \rangle$. We subject the system to the coordinate transformation
\begin{align}\label{localcoordinates} 
x_i = \sqrt{m_i}q_i,\ w_i = \frac{p_i}{\sqrt{m_i}},
\end{align} 
and obtain the Hamiltonian 
$H(x,w) = \frac{1}{2}\langle w,w \rangle + \langle c_2,x \rangle$, 
$c_2 = (\sqrt{m_1}, \sqrt{m_2}, \sqrt{m_3})$. The natural Riemannian metric in these 
coordinates is the standard Euclidean inner product. The new generators of length one are
\begin{align}\label{h}
h_1 = \frac{1}{\sqrt{M_1}}
\begin{pmatrix}
\sqrt{m_1}\\
\sqrt{m_2}\\
\sqrt{m_3}
\end{pmatrix},
h_2 = \frac{1}{\sqrt{M_2}}
\begin{pmatrix}
0\\
\sqrt{m_2}\\
\sqrt{m_3}
\end{pmatrix},
h_3 = 
\begin{pmatrix}
0\\
0\\
1
\end{pmatrix},
\end{align}
where $M_i = m_i + \cdots + m_3$, $i = 1,2$. The configuration space changes to 
\begin{align}\label{Wx}
W_x(h_1,h_2,h_3) = \{(x_1,x_2,x_3) \in \mathbb{R}^3:\ 0 \leq \frac{x_1}{\sqrt{m_1}} 
\leq \frac{x_2}{\sqrt{m_2}} \leq \frac{x_3}{\sqrt{m_3}}\}.
\end{align}
With respect to the Euclidean inner product we have 
\begin{align}
\langle h_i, h_j\rangle = \frac{\sqrt{M_j}}{\sqrt{M_i}},\ 1 \leq i < j \leq 3,\nonumber
\end{align}
which immediately yields properties (\ref{point4}), (\ref{point5}) from Proposition \ref{prop2.3}, 
proving that $W_x(h_1,h_2,h_3)$ is a simple wedge. Further, using properties (\ref{2.1}) and 
(\ref{2.2}) we get a direct link between the angles 
characterizing the wedge and the masses of the FB system
\begin{align}\label{3.4}
\cos^2 \alpha_i = \frac{M_{i+1}}{M_i},\ \sin^2 \alpha_i = \frac{m_i}{M_i},\ 
\tan^2 \beta_i = \frac{m_i}{m_{i+1}}.
\end{align}
Notice, that the direction of the 
acceleration vector is along the first generator. We arrived at the important conclusion, that a PW 
system in a simple wedge with acceleration vector along the first generator is equivalent to a FB 
system with appropriate masses.

\subsection{Wide wedges}
\begin{definition}
A three dimensional wedge with generators $\{g_1,g_2,g_3\}$ is wide if the angle of the generators 
exceeds $\pi/2$, i.e. $\langle g_i,g_j \rangle < 0,$ $1 \leq i < j \leq 3$.
\end{definition}

Consider a PW system in a simple wedge $W_x(h_1,h_2,h_3)$ (\ref{Wx}). 
We will unfold $W_x(h_1,h_2,h_3)$ to a wide wedge by continuously reflecting 
it in the faces, which 
are equipped with the first generator, i.e. $W(h_1,h_2)$ and $W(h_1,h_3)$.  
It is not hard to see, that this 
procedure creates a wide wedge if and only if the angle between the subspaces $S(h_1,h_2)$ 
and $S(h_1,h_3)$ is exactly\footnote{Otherwise the unfolded simple wedges would overlap.} 
$\pi/3$. This translates to the condition
\begin{align}\label{angle}
\frac{1}{2} = \cos \frac{\pi}{3} = \langle n_{S(h_1,h_2),0}, n_{S(h_1,h_3),0} \rangle,
\end{align}
where $n_{S(h_1,h_2),0}$ resp. $n_{S(h_1,h_3),0}$ are the unit normal vectors of the 
subindexed subspace. Using (\ref{h}) in (\ref{angle}) we obtain for the appropriate masses of the 
corresponding FB system 
\begin{align}\label{masses}
2\sqrt{m_1}\sqrt{m_3} = \sqrt{m_1+m_2}\sqrt{m_2+m_3}.
\end{align}
In this way we obtain new generators $\{g_1,g_2,g_3\}$ and the wedge $W_x(g_1,g_2,g_3)$, 
which consists exactly of six simple wedges. 
With the help of (\ref{h}) and elementary linear algebra it follows rather easily that the wedge 
$W_x(g_1,g_2,g_3)$ is wide. 

The two dimensional inner faces of the simple wedges possessing the first generator $h_1$ correspond 
to a collision of two balls in the associated FB system. 
When the particle hits one of the inner faces we allow the particle to pass through the face to the 
adjacent wedge. 

A collision of the particle with one of the faces of the wide wedge 
corresponds to a collision with the floor in the 
associated FB system. In this case, we do not allow the particle to pass through the face, 
but instead reflect the velocity vector across the face by using $w_1^+=-w_1^-$. 

Since the trajectory is parabolic, a natural question to ask is, whether or not grazing collisions 
can occur. For our purposes we will confine ourselves to the simple wedge $W_x(h_1,h_2,h_3)$. 
The definition of a grazing collision is as follows
\begin{definition}
A collision of the trajectory $x(t)$, at time $t_0$, with one of the faces of the simple wedge 
$W_x(h_1,h_2,h_3)$ is grazing, if the velocity vector $\dot{x}(t_0)$ lies in the face of 
collision.
\end{definition}
The next result gives equivalent conditions of a grazing collision 
with one of the faces possessing the first generator.
\begin{prop}\label{4}
	Let $t_0 < t_1$ be consecutive collision times of the trajectory in the simple wedge 
	$W_x(h_1,h_2,h_3)$ 
	and assume that $x(t_1) \in W_x(h_1,h_2)$ or $x(t_1) \in W_x(h_1,h_3)$.
The following statements are equivalent:
\begin{align}
	&\text{1. A collision with the face } W_x(h_1,h_2)\ \text{resp. } W_x(h_1,h_3), 
	\text{ at time } t_1, \text{ is grazing.}
	\nonumber\\
	&\text{2. The differences } \frac{w_1^+(t_0)}{\sqrt{m_1}} - \frac{w_2^+(t_0)}{\sqrt{m_2}}\ \text{
	resp. } \frac{w_2^+(t_0)}{\sqrt{m_2}} - \frac{w_3^+(t_0)}{\sqrt{m_3}}\ \text{are equal to zero}.
	\nonumber\\
	&\text{3. The trajectory segment } \{x(t): t \in [t_0,t_1]\}\ \text{is confined to } W_x(h_1,h_2)\ 
	\text{resp. }	W_x(h_1,h_3).\nonumber
\end{align}
\end{prop}

\begin{proof}
	$1 \Rightarrow 2$:\\
 Without loss of generality assume that 
$x(t_0) \in W_x(h_1,h_3)$ or $x(t_0) \in W_x(h_2,h_3)$. Further, let the particle experience 
a grazing collision with the face $W_x(h_1,h_2)$ at time $t_1$. In a grazing collision 
the velocity
\begin{align}
w^-(t_1) =
\begin{pmatrix}
-\sqrt{m_1}(t_1 - t_0) + w_1^+(t_0) \\
-\sqrt{m_2}(t_1 - t_0) + w_2^+(t_0) \\
-\sqrt{m_3}(t_1 - t_0) + w_3^+(t_0) 
\end{pmatrix}\nonumber
\end{align}
is parallel to the face 
\begin{align}
	W_x(h_1,h_2) = \{(x_1,x_2,x_3) \in W_x(h_1,h_2,h_3):\ \frac{x_1}{\sqrt{m_1}} = \frac{x_2}{\sqrt{m_2}}\}.
	\nonumber
\end{align}
This is equivalent to 
\begin{align}
\frac{w_1^+(t_0)}{\sqrt{m_1}} = \frac{w_2^+(t_0)}{\sqrt{m_2}}.\nonumber
\end{align}
The argument for a grazing collision with the face $W_x(h_1,h_3)$ is exactly the same.\\
$2 \Rightarrow 3$:\\
Without loss of generality assume again that 
$x(t_0) \in W_x(h_1,h_3)$ or $x(t_0) \in W_x(h_2,h_3)$ and let the particle collide with the face 
$W_x(h_1,h_2)$ at time $t_1$.
From the Hamiltonian equations, we calculate the first collision time
\begin{align}
t_1 - t_0 = \frac{x_2(t_0) / \sqrt{m_2} - x_1(t_0) / \sqrt{m_1}}
{w_1^+(t_0) / \sqrt{m_1} - w_2^+(t_0) / \sqrt{m_2}}.\nonumber
\end{align}
Since the energy is fixed, $t_1 - t_0 < \infty$. It follows, that if 
$w_1^+(t_0) / \sqrt{m_1} - w_2^+(t_0) / \sqrt{m_2} \to 0$, then 
$x_2(t_0) / \sqrt{m_2} - x_1(t_0) / \sqrt{m_1} \to 0$ (at least) with the same rate. 
Thus, in case of equal velocities, we always have 
$x_1(t_0) / \sqrt{m_1} = x_2(t_0) / \sqrt{m_2}$, which implies that the trajectory moves inside 
the face $W_x(h_1,h_2)$. \\
The argument for $w_2^+(t_0) / \sqrt{m_2} - w_3^+(t_0) / \sqrt{m_3} = 0$ is exactly the same.\\
$3 \Rightarrow 1$:\\
This direction is immediate.
\end{proof}

\subsection{Projection} The Hamiltonian equations imply that the flow is an inverted parabola. 
Let $[t_0,t_c]$ be the time from one collision to the next. 
We define the planar subspace
\begin{align}\label{planar}
\mathsf{P}_{x([t_0,t_c])} = S(\dot{x}(t_1),\dot{x}(t_2)),\ 
\dot{x}(t_1) \neq \dot{x}(t_2),\ t_0 \leq t_1 < t_2 \leq t_c.
\end{align}
The movement of the parabolic trajectory is confined to the planar subspace, i.e.
\begin{align}
\{x(t):\ t \in [t_0,t_c]\} \subset \mathsf{P}_{x([t_0,t_c])}.\nonumber
\end{align}
The acceleration vector $a =  \ddot{x}(t)$ is always element of 
$\mathsf{P}_{x([t_0,t_c])}$: Set
%$S(\dot{x}(t_1),\dot{x}(t_2))$: 
\begin{align}
n_p(t,t_1) = \dot{x}(t) \times \dot{x}(t_1),\ \dot{x}(t) \neq \dot{x}(t_1).\nonumber
\end{align}
Since the trajectory moves inside a planar subspace, $n_p(t,t_1)$ is constant for all choices 
$t, t_1 \in [t_0, t_c]$, $t \neq t_1$. Thus, $\dot{n}_p(t,t_1) = 0$.
Observe, that 
\begin{align}\label{normal}
\langle n_{x(t)}, \dot{x}(t) \rangle = 0,\ \forall\ t \in [t_0,t_c],
\end{align}
where $n_{x(t)}$ is a normal vector to $\dot{x}(t)$ at point $x(t)$. Differentiating (\ref{normal}) 
gives
\begin{align}
\langle n_{x(t)}, \ddot{x}(t) \rangle = -\langle \dot{n}_{x(t)}, \dot{x}(t) \rangle.\nonumber
\end{align}
Substituting $\ddot{x}(t)$ with $a$ and $n_{x(t)}$ with $n_p(t,t_1)$ gives
\begin{align}
 \langle a, n_p(t,t_1)\rangle = - \langle \dot{n}_p(t,t_1), \dot{x}(t) \rangle = 0.\nonumber
\end{align}

We will use this fact to project the configuration space $W_x(g_1,g_2,g_3)$ along the first generator 
$h_1$ 
to the plane spanned by the normal vectors $n_{S(h_1,h_2)}$, $n_{S(h_1,h_3)}$ of the subspaces 
$S(h_1,h_2)$, $S(h_1,h_3)$. The projected configuration space becomes an 
equilateral triangle. Its algebraic form is 
given by 
\begin{align}\label{triangle}
\bigtriangleup:\ \sqrt{m_1}x_1 + \sqrt{m_2}x_2 + \sqrt{m_3}x_3 = d,\ d > 0,
\end{align}
where $d$ determines its displacement from the origin. Since the acceleration vector lies in the 
plane spanned by two velocity vectors of the flow, the parabola projected to $\bigtriangleup$ becomes a 
straight line (see Figure 1).\\
\begin{center}
\begin{tikzpicture}[scale=2]
	\draw (0,0) -- (2,0);
	\draw (0,0) -- ($(1,{sqrt(3)})$);
	\draw (2,0) -- ($(1,{sqrt(3)})$);
    \draw (0,0) -- ($(1.5,{sqrt(3)*0.5})$) node [right] {$h_2$};
    \draw (2,0) -- ($(0.5,{sqrt(3)*0.5})$) node [left] {$h_3$};
    \draw ($(1,{sqrt(3)})$) -- (1,0) node [below] {$h_4$};
    \draw[dashed,->] (0.4,0) -- ($(0.8,{sqrt(3)*0.8})$) -- (0.9,1.1);
    \node[below left] at (0,0) {$g_3$};
    \node[below right] at (2,0) {$g_1$};
    \node[above] at ($(1,{sqrt(3)})$) {$g_2$};
    \node [above right] at (0.97,0.65) {$h_1$};
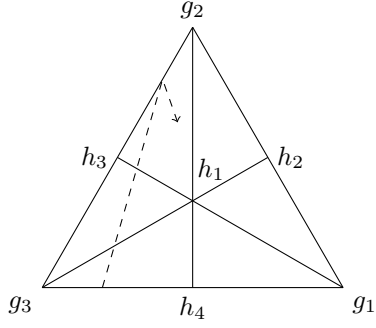
    \node [align=left, below right] at (-0.2,-0.5) {Figure 1: The projected parabola\\ moving inside the projected\\ 
    	configuration space $\bigtriangleup$.};
\end{tikzpicture}
\qquad
\begin{tikzpicture}[scale=2]
\draw (0,0) -- (2,0);
\draw (0,0) -- ($(1,{sqrt(3)})$);
\draw (2,0) -- ($(1,{sqrt(3)})$);
\draw (0,0) -- ($(1.5,{sqrt(3)*0.5})$) node [right] {$(2,3)$};
\draw (2,0) -- ($(0.5,{sqrt(3)*0.5})$) node [left] {$(2,3)$};
\draw ($(1,{sqrt(3)})$) -- (1,0) node [below] {$(2,3)$};
\node[below left] at (0,0) {$(1,2)$};
\node[below right] at (2,0) {$(1,2)$};
\node[above] at ($(1,{sqrt(3)})$) {$(1,2)$};
%\draw [red] (0.8,0) -- ($(0.4,{sqrt(3)*0.4})$);
\draw [blue] (0.8,0) -- ($(0.7,{sqrt(3)*0.7})$) node [above left] {I.};
\draw [cyan] (0.8,0) -- ($(1.1,{-sqrt(3)*1.1+2*sqrt(3)})$) node [above right] {II.};
\draw [green] (0.8,0) -- ($(1.4,{-sqrt(3)*1.4+2*sqrt(3)})$) node [above right] {III.};
\draw [teal] (0.8,0) -- ($(1.7,{-sqrt(3)*1.7+2*sqrt(3)})$) node [right] {IV.};
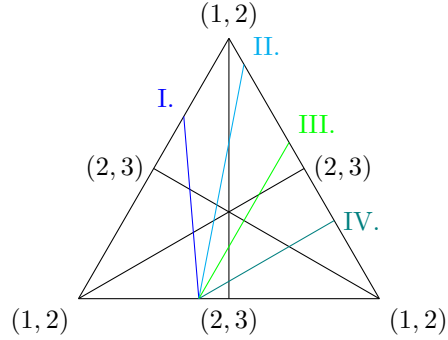
\node [align=left, below right] at (0,-0.5) {Figure 2: An example of \\ cases I-IV.\\};
\end{tikzpicture}
\end{center}

\subsection{Proper alignment in wide wedges}\label{special}
The idea to unfold the simple wedge $W_x$ (\ref{Wx}) into a wide wedge stems from Wojtkowski \cite{W16}. It is evident, that 
the triple collision states in the configuration space, which are represented by the first generator $h_1$, disappear in the 
wide wedge. More precise, each trajectory, which passes through the spot where $h_1$ was, has a smooth continuation. 
Since the triple collision singularity manifold is the only obstacle in proving the proper alignment condition, the system 
of a particle falling in the wide wedge, obtained for the special mass configuration (\ref{masses}), satisfies 
the proper alignment condition. However, in the simple wedge $W_x$, once a trajectory hits the corner $h_1$ it is impossible 
to continue it uniquely, since it has two images after the singular collision. The latter holds for any possible mass 
configuration. Therefore, the validity of the proper alignment condition cannot be immediately deduced from the dynamics of the 
wide wedge. It remains unknown at the moment (see Subsection \ref{state} for more details).

\section{Strict unboundedness - Part II}\label{eight}
Consider a PW system in the simple wedge $W_x(h_1,h_2,h_3)$ (\ref{Wx}) and mass 
restrictions given by (\ref{masses}). Due to the results of the last section we reflect the simple 
wedge in its faces possessing the first generator to obtain a wide wedge $W_x(g_1,g_2,g_3)$. 

For the strict unboundedness, it remains to prove Theorem \ref{thm2} and 
property (\ref{point3}) from Section 6. The latter was already proven as part of the Main Theorem 6.6 
in \cite[p. 327 - 331]{W98}. In essence, Wojtkowski proved, that every orbit will eventually hit every 
face of the wide wedge. Subsequently, this yields all necessary collisions for eventually mapping 
the Lagrangian subspaces $L_1$ and $L_2$ inside the interior of the contracting cone 
field\footnote{One can directly calculate, that all $(\delta \xi, 0) \in L_1$ get mapped into $\mathcal{C}(x)$ after at 
most three returns to the floor and all $(\delta \eta, 0) \in L_2$ as soon as the trajectory 
experiences the first two ball to ball collisions.}.

To prove Theorem \ref{thm2} we first establish how many different collisions, involving all the 
balls, are possible in between two consecutive collisions of the lowest ball with the floor. 
Using the projection to $\bigtriangleup$ (see (\ref{triangle})), we encounter the following four different possibilities 
(see Figure 2)
\begin{align}
\begin{array}{ll}\label{cases}
\text{I.} &(0,1) \longrightarrow (1,2) \longrightarrow (2,3) \longrightarrow (0,1)\\[0.1cm]
\text{II.} &(0,1) \longrightarrow (1,2) \longrightarrow (2,3) \longrightarrow (1,2) \longrightarrow (0,1)
\\[0.1cm]
\text{III.} &(0,1) \longrightarrow (2,3) \longrightarrow (1,2) \longrightarrow (0,1)\\[0.1cm]
\text{IV.} &(0,1) \longrightarrow (2,3) \longrightarrow (1,2) \longrightarrow (2,3) \longrightarrow (0,1)
\end{array}
\end{align}

\begin{proof}[Proof of Theorem \ref{thm2}]
First observe, that since every collision in the FB system happens infinitely often, at least one of 
the cases I-IV (\ref{cases}) occurs infinitely often, in every orbit of the system. Due to symmetry 
it is enough to consider only the first two cases. Without loss of generality we start at time $t_0$ 
on the face $W_x(g_3,h_4)$. In case I, the order of faces crossed by the trajectory is 
$W_x(h_1,g_3)$, $W_x(h_1,h_3)$ before the particle hits the last face $W_x(h_3,g_2)$. In case II, 
the trajectory crosses faces $W_x(h_1,g_3)$, $W_x(h_1,h_3)$, $W_x(h_1,g_2)$ before it reaches the last 
face $W_x(h_2,g_2)$. We compactly display the latter information as
\begin{align}
\begin{array}{ll} 
\mathbf{Case\ I.} & W_x(g_3,h_4) \longrightarrow W_x(h_1,g_3) \longrightarrow W_x(h_1,h_3) 
\longrightarrow W_x(h_3,g_2),\nonumber\\[0.2cm]
\mathbf{Case\ II.} & W_x(g_3,h_4) \longrightarrow W_x(h_1,g_3) \longrightarrow W_x(h_1,h_3) 
\longrightarrow W_x(h_1,g_2) \longrightarrow W_x(h_2,g_2).\nonumber
\end{array}
\end{align}
\textbf{Case I.} Let $t_1 < t_2 < t_3$, 
be the collision times with the faces $W_x(h_1,g_3)$, $W_x(h_1,h_3)$ and $W_x(h_3,g_2)$. 
When the particle crosses the face $W_x(h_1,g_3)$ resp. $W_x(h_1,h_3)$, we have
\begin{align}\label{difference}
\frac{w_1^-(t_1)}{\sqrt{m_1}} - \frac{w_2^-(t_1)}{\sqrt{m_2}} > 0\quad \text{resp.}\quad 
\frac{w_2^-(t_2)}{\sqrt{m_2}} - \frac{w_3^-(t_2)}{\sqrt{m_3}} > 0.
\end{align}
The velocity differences are invariant in between collision, i.e.
\begin{align}\label{invariant}
\begin{array}{rcl}
\dfrac{w_1^-(t_1)}{\sqrt{m_1}} - \dfrac{w_2^-(t_1)}{\sqrt{m_2}} &=& 
\dfrac{w_1^+(t_0)}{\sqrt{m_1}} - \dfrac{w_2^+(t_0)}{\sqrt{m_2}},\\[0.3cm]
\dfrac{w_2^-(t_2)}{\sqrt{m_2}} - \dfrac{w_3^-(t_2)}{\sqrt{m_3}} &=& 
\dfrac{w_2^+(t_1)}{\sqrt{m_2}} - \dfrac{w_3^+(t_1)}{\sqrt{m_3}},
\end{array}
\end{align}
Due to Proposition \ref{4}, the quantities (\ref{difference}) are arbitrarily close to zero if and
only if the collisions with the respective faces are arbitrarily close to grazing ones.
The first collision with the face $W_x(h_1,g_3)$ is almost grazing if and only if 
the planar subspace $\mathsf{P}_{x([t_0,t_3])}$ (see (\ref{planar})) is almost perpendicular to 
the face $W_x(g_1,g_2)$, i.e. $x(t_3) \in W_x(g_1,g_2)$. But this contradicts the fact of 
the trajectory reaching the last face $W_x(h_3,g_2)$. Therefore, there exists $\psi_1 > 0$, such that 
for all $x(t_0) \in W_x(h_4,g_3)$:
\begin{align}
	\sphericalangle(\mathsf{P}_{x([t_0,t_3])},W_x(g_3,h_1)) > \psi_1.
\end{align}

The second collision with the face $W_x(h_1,h_3)$ is almost grazing if and only if 
$\mathsf{P}_{x([t_0,t_c])}$ is almost perpendicular to the face $W_x(g_2,g_3)$, i.e. 
$x(t_0) \in W_x(g_1,h_4)$. 
But this contradicts $x(t_0) \in W_x(h_4,g_3)$. Therefore, there exists $\psi_2 > 0$, such that for all 
$x(t_0) \in W_x(h_4,g_3)$:
\begin{align}
	\sphericalangle(\mathsf{P}_{x([t_0,t_3])},W_x(h_1,h_3)) > \psi_2.
\end{align}
Using the projection along the first generator (see (\ref{triangle}) and Figure 1) we conclude, that 
$\psi_1 = \psi_2 = \pi / 6$.

\textbf{Case II.} Let $t_1 < t_2 <t_3 < t_4$ be the collision times of the particle 
with the faces $W_x(h_1,g_3)$, $W_x(h_1,h_3)$, $W_x(h_1,g_2)$ and $W_x(h_2,g_2)$. 
It is sufficient to prove that either
\begin{align}\label{first}
\frac{w_1^-(t_1)}{\sqrt{m_1}} - \frac{w_2^-(t_1)}{\sqrt{m_2}}\quad \text{and}\quad 
\frac{w_2^-(t_2)}{\sqrt{m_2}} - \frac{w_3^-(t_2)}{\sqrt{m_3}}
\end{align}
or
\begin{align}\label{second}
\frac{w_2^-(t_2)}{\sqrt{m_2}} - \frac{w_3^-(t_2)}{\sqrt{m_3}}\quad \text{and}
\quad 
\frac{w_1^-(t_3)}{\sqrt{m_1}} - \frac{w_2^-(t_3)}{\sqrt{m_2}}
\end{align}
are uniformly bounded away from zero. 

In order to reach the last face $W_x(g_2,h_2)$, the quantity 
$w_2^-(t_2) / \sqrt{m_2} - w_3^-(t_2) / \sqrt{m_3}$ is always uniformly bounded 
away from zero. Otherwise, due to Proposition \ref{4},
$\mathsf{P}_{x([t_0,t_4])}$ would be perpendicular to the face $W_x(g_2,g_3)$ and, thus, 
never reach the last face $W_x(h_2,g_2)$.

Due to Proposition \ref{4}, $w_1^-(t_1) / \sqrt{m_1} - w_2^-(t_1) / \sqrt{m_2}$ is arbitrarily close 
to zero if and only if the planar subspace $\mathsf{P}_{x([t_0,t_4])}$ is 
almost perpendicular to the face $W_x(g_1,g_2)$. But this implies that 
$w_1^-(t_3) / \sqrt{m_1} - w_2^-(t_3) / \sqrt{m_2}$ is uniformly bounded away from zero.

If $w_1^-(t_3) / \sqrt{m_1} - w_2^-(t_3) / \sqrt{m_2}$ is arbitrarily 
close to zero, then by the same reasoning as above, 
$w_1^-(t_1) / \sqrt{m_1} - w_2^-(t_1) / \sqrt{m_2}$ is uniformly bounded away from zero.
Thus, in case II., either (\ref{first}) or (\ref{second}) are always uniformly 
bounded away from zero.

It is clear, due to the coordinate transformation (\ref{localcoordinates}), that 
$w_i / \sqrt{m_i} - w_{i+1} / \sqrt{m_{i+1}}$ is uniformly bounded from below if and only if 
$v_i - v_{i+1}$ is uniformly bounded from below.

To finish the proof, consider the FB system in $x = (\xi,\eta)$ coordinates.
Along every orbit $(T^nx)_{n \in \mathbb{N}}$ we have obtained two subsequences
$(T^{n_{2k}}x)_{k \in \mathbb{N}}$ and $(T^{n_{2k+1}}x)_{k \in \mathbb{N}}$, where we set
$(T^{n_{2k}}x)_{k \in \mathbb{N}}$ to be the phase points before- and 
$(T^{n_{2k+1}}x)_{k \in \mathbb{N}}$ right after, two consecutive collisions with velocity differences
bounded away from zero. This means, that the derivative map $dT^{n_{2k-1}-n_{2k-2}}$ equals either 
$d\Phi_{2,3}d\Phi_{1,2}$ or $d\Phi_{1,2}d\Phi_{2,3}$. Both of the latter maps are upper triangular 
matrices of the form 
\begin{align}
\begin{pmatrix}
X_1 & X_2\\
0 & X_1^T
\end{pmatrix}.\nonumber
\end{align}
$X_1$ depends only on the masses, while $X_2 = X_2(\alpha_1,\alpha_2)$ depends on the masses 
and the velocity differences $v_i - v_{i+1}$ in $\alpha_1,\alpha_2$ (see (\ref{alpha})).
Each pair of consecutive collisions with velocity differences bounded away from 
zero belongs to one of the cases I-IV (\ref{cases}). Each of these velocity differences has a 
uniform lower bound. Set the minimum of these lower bounds to be $\Theta > 0$. Observe, that
\begin{align}
Q(d_{T^{n_{2k-2}}x}T^{n_{2k-1}-n_{2k-2}}(0,\delta \eta)) = 
\langle X_2\frac{1}{\|(0,\delta \eta)\|}\delta \eta, 
X_1^T \frac{1}{\|(0,\delta \eta)\|} \delta \eta \rangle
\|(0,\delta \eta)\|^2.\nonumber
\end{align}
Let $X_2(\Theta)$ be the matrix in which the velocity differences in $X_2(\alpha_1,\alpha_2)$ are 
replaced by $\Theta$. Since $X_2(\alpha_1,\alpha_2) - X_2(\Theta) > 0$, we have
\begin{align}
\langle X_2(\alpha_1,\alpha_2)\frac{1}{\|(0,\delta \eta)\|}\delta \eta, 
X_1^T \frac{1}{\|(0,\delta \eta)\|} \delta \eta \rangle
> \langle X_2(\Theta)\frac{1}{\|(0,\delta \eta)\|}\delta \eta, 
X_1^T \frac{1}{\|(0,\delta \eta)\|} \delta \eta \rangle.\nonumber
\end{align}
The functional $f(u) = \langle X_2(\Theta)u, X_1^Tu\rangle$ is positive, independent 
of $x$ and continuous on the compact space $\partial B_{\|\cdot\|}(0,1)$. Thus, there exists a 
constant $\Lambda > 0$ such that
\begin{align}
Q(d_{T^{n_{2k-2}}x}T^{n_{2k-1}-n_{2k-2}}(0,\delta \eta)) > \Lambda \|(0,\delta \eta)\|^2.\nonumber
\end{align}
This finishes the proof of Theorem \ref{thm2} and therefore also Theorem \ref{thm1}. 
\end{proof}

As it was outlined in Section 2, the strict unboundedness for every orbit subsequently implies the 
Chernov-Sinai ansatz and the abundance of least expanding points.

\end{document}